\documentclass[12pt]{amsart}
\usepackage[margin=0.7in]{geometry}
\usepackage{amssymb,amsthm,amsfonts}
\usepackage{mathtools}
\usepackage{graphicx,epstopdf,color}
\usepackage{array}

\usepackage{hyperref}
\usepackage[alphabetic]{amsrefs}

\usepackage{amsmath}
\allowdisplaybreaks
\usepackage[latin1]{inputenc}
\newcommand{\R}{\mathbb{R}}
\newcommand{\Z}{\mathbb{Z}}
\newcommand{\N}{\mathbb{N}}

\newtheorem{thm}{Theorem}[section]
\newtheorem{cor}[thm]{Corollary}

\newtheorem{rem}[thm]{Remark}

\renewcommand{\leq}{\leqslant}
\renewcommand{\le}{\leqslant}
\renewcommand{\geq}{\geqslant}
\renewcommand{\ge}{\geqslant}

\renewcommand{\epsilon}{\varepsilon}

\numberwithin{equation}{section}

\title[Stationary distributions and anomalous diffusion]{The L\'evy flight foraging hypothesis:
\\ comparison between
stationary distributions \\
and anomalous diffusion}\thanks{This research has been supported by the
Australian Laureate Fellowship FL190100081 ``Minimal
surfaces, free boundaries and partial differential equations''. 
The authors declare that they have no known competing financial interests
that could have appeared to influence the work reported in this paper.
It is a pleasure to thank Kurt Williams for interesting discussions.
We also gratefully thank the Referees for their constructive comments and recommendations.}

\author{Serena Dipierro}
\address{Serena Dipierro, Department of Mathematics and Statistics,
University of Western Australia, 35 Stirling Highway,
WA6009 Crawley, Australia}\email{serena.dipierro@uwa.edu.au}

\author{Giovanni Giacomin}
\address{Giovanni Giacomin, Department of Mathematics and Statistics,
University of Western Australia, 35 Stirling Highway,
WA6009 Crawley, Australia}
\email{giovanni.giacomin@research.uwa.edu.au}

\author{Enrico Valdinoci}
\address{Enrico Valdinoci, Department of Mathematics and Statistics,
University of Western Australia, 35 Stirling Highway,
WA6009 Crawley, Australia}
\email{enrico.valdinoci@uwa.edu.au}


\begin{document}

\maketitle

\begin{abstract} We consider a stationary prey in a given region of space and
we aim at detecting optimal foraging strategies.

On the one hand, when the prey is uniformly distributed, the best possible strategy for the forager is to be stationary and uniformly distributed in the same region.

On the other hand, in several biological settings, foragers cannot be completely stationary,
therefore we investigate the best seeking strategy for L\'evy foragers in terms of the corresponding L\'evy exponent. In this case, we show that the best strategy depends on the region size in which the prey is located:
large regions exhibit optimal seeking strategies close to Gaussian random walks, while
small regions favor L\'evy foragers with small fractional exponent.

We also consider optimal strategies in view of the Fourier transform of the distribution of a stationary prey. When this distribution is supported in a suitable volume, then the foraging efficiency functional is monotone increasing with respect to the L\'evy exponent
and accordingly the optimal strategy is given by the Gaussian dispersal.
If instead the Fourier transform of the distribution of a stationary prey is supported in the complement of a suitable volume, then the foraging efficiency functional is monotone decreasing
with respect to the L\'evy exponent
and therefore the optimal strategy is given by a null fractional exponent (which in turn
corresponds, from a biological standpoint, to a strategy of ``ambush'' type).

We will devote a rigorous quantitative analysis also to emphasize some specific differences between the
one-dimensional and the higher-dimensional cases.
\end{abstract}

\section*{Preamble: the L\'evy flight foraging hypothesis for a general audience}

Understanding how animals search for food in their environment is not only important for our knowledge of nature and animal behavior but it is also useful for suggesting optimal seeking strategies, for instance, for robots and internet engines.

The classical description of animal foraging relied on ``random walks'': an animal looking for something to eat in a big area, such as a forest, unaware of the details of the environment, was supposed to take a small step in a random direction and eat the food found in this way.

However, ``longer'' steps may provide a better strategy. After all, on many occasions, the preys do not sit there waiting for being eaten, therefore a predator may have better hunting chances through a ``hit-and-go'' procedure.
Also, preys may occur in sparse patches, therefore hunting strategies that are too ``localized'' may end up revisiting too often the same region, not allowing for a sufficiently fast exploration of the territory.

In this scenario, the L\'evy flight foraging hypothesis suggested that sometimes it may be better for the forager to take a few really long steps, almost like jumping, instead of taking lots of short steps. 

In jargon, these long jumps are called ``L\'evy flights'' and they may be useful to cover ``more ground'' quickly: animals which thereby explore a larger area in less time may be more ``efficient'' in finding food. Moreover, strategies of this sort could help animals adapt to different environments, adjusting to different conditions and prey distributions, and finally enhancing their chance of survival.

The L\'evy flight foraging hypothesis is however rather controversial and it is very difficult to have ultimate answers for such a complex topic, involving so many parameters and so various conditions.

When things get difficult, some mathematics can be helpful, also to isolate different kinds of complications, to provide simple, but unambiguous models, to provide some exact answer under precisely stated and possibly verifiable hypotheses.

In this paper, we investigate the L\'evy flight foraging hypothesis looking at several main scenarios.

When prey is evenly spread out in a given area, the best strategy for the predator is to remain in that area and search around.

But in some cases, animals cannot just stay in one place. They need to move around to search for food or for a number of other activities (such as finding a partner). In these situations, in large areas, it may be better for foragers to move in a way that is similar to random walking, but in smaller areas the L\'evy pattern may be preferable.

Analytical methods, such as the Fourier transform, can be useful to describe the distribution of prey and detect where it is concentrated. If prey is mostly concentrated in a certain area, it appears that random walks are preferable, but if prey is mostly found outside of that area very long jumps may become optimal, also in view of ``ambush'' strategies.

The dimension of the environment also plays an important role.

\section{Introduction}\label{1-Sljn3hfig}

A classical topic in mathematical biology consists in the detection of optimal strategies for a forager in search of food. This problem has been broadly investigated in the recent literature also in relation with the so-called L\'evy flight foraging hypothesis, according to which L\'evy flights, and in particular the ones with inverse square distribution of flight times or distances, can optimize search efficiencies for sparsely and randomly distributed revisitable preys in the absence of memory, reducing the oversampling associated with the classical Brownian motion. On the one hand, over the years experimental evidence has supported the occurrence of L\'evy flights in animal environment (see for example~\cite{101242jeb009563}, in which the patterns of foraging honeybees were traced using harmonic radar
and detecting resemblance with L\'evy-flights,
\cite{SIMSS},
in which diverse marine predators, such as sharks, bony fishes, sea turtles, and penguins, have been demonstrated to exhibit a suitable L\'evy type behavior, 
\cite{ARFGBHAR}, 
which proposed that the motility of some lymphocytes in the brain is accurately described by a L\'evy distribution, also enhancing their capability of finding rare targets,
\cite{CHEN},
which showed how intracellular trafficking self-organize into L\'evy patterns,
\cite{ARIEL}, in which dense swarms of bacteria
have been shown to perform a L\'evy type superdiffusion, 
etc.). Additionally, L\'evy patterns seem to occur also in human activities
(see e.g.~\cite{FLLOW} which showed that
geographical movements of US banknotes display dispersal relations similar to L\'evy flights).

On the other hand, the L\'evy flight foraging hypothesis is still a highly debated and somewhat controversial topic and a precise and rigorous approach appears to be essential in stating exactly the hypothesis and in describing carefully the several biological phenomena related to it (see e.g.~\cite{MR4071775, MR4207853, MR4207854, Xkferwupoj54t-0olr, CLEME}; further suggestions and disagreements related to
L\'evy flights will be discussed in Section~\ref{DIAsdff}). See~\cite{RAYNS} for a detailed review.

In this context, we have recently introduced and discussed several ``efficiency functionals'' related to animal foraging and motivated by the L\'evy flight foraging hypothesis, see~\cite{EFLFFH, ULTI, DGV1, DGV1-1, SPETTRALE}. In a nutshell, the gist of these functionals is to consider the distribution of foragers and preys as probability densities and evaluate the outcome of the search from the random encounters of foragers and preys, as encoded by the product of the corresponding distributions. In addition, one can account for the energy spent by the forager during the search, by considering it as a penalization gauging the time needed for the hunt or the distance traveled by the forager. {F}rom the biological point of view, one may also consider negative values of the distribution of preys, e.g. to account for a lack of resources or for regions containing poisonous food.\medskip

A simplified, but biologically interesting, scenario is that of stationary preys: in this case the distribution of prey does not depend on time. In this setting, one could also consider the case of stationary foragers, or foragers with the ability of choosing the foraging strategy of remaining stationary. We stress however
that several biological scenarios exhibit foragers which can never remain stationary. For instance, some species of sharks (white sharks, bull sharks, sandbar sharks, porbeagles, salmon sharks, thresher sharks, pelagic threshers, bigeye threshers, alopias vulpinus, mako sharks, hammerheads, winghead sharks, bonnetheads, scoopheads, etc.) are considered obligate ram ventilators and must keep moving in order to receive oxygen from the water passing through their gills. Similarly, some species of whales, dolphins, tunas, billfishes, swordfishes, bonitos, scombrids and bluefishes need to keep moving. See e.g.~\cite{92-20082wqfer157967, 11-9oijhknfOLIjndwkv, 13-20082wqfer, 19Xo0jlfj567890}.
Some species of lizards, gerbils and turtles are also known to be very actively moving animals, see e.g.~\cite{77-kpXmvd2891937, 00-jn200083ojwflb3123, 2004THGIUDKNF}. Drifters in constantly moving waters, such as some jellyfishes, or parasites, such as cyamids  attached to constantly moving cetaceans, can be thought as seekers in continuous motions (see e.g.~\cite{16-wkolmnbXbgmflmn, Hadfield2019}).  Also, in principle, the capacity of detecting adjacent preys could be influenced by the length and tempo of the forager's attention span (see e.g.~\cite{BRUSJDN, oaskcndATT} for some measurements of these quantities), hence the ansatz that stationary foragers are always capable to intercept nearby preys would also require some specific validation. Therefore it is biologically very interesting to compare 
the efficiency of foraging strategies of stationary and diffusive foragers in different environmental scenarios and to provide a rigorous mathematical framework to address these types of problems.
\medskip

The results that we obtain in this paper deal with a suitable efficiency functional (which will be made precise below) and with its optimization in terms of a L\'evy flight fractional exponent~$s\in(0,1)$ (the case~$s=1$ corresponding to Gaussian diffusion). Surfing over technicalities, our results\footnote{The numbering of these theorems matches the corresponding ones presented in Section~\ref{forma} in a formal mathematical setting.
The biological interpretation of these results has not to be taken literally, and several
limitations will be discussed in Section~\ref{DIAsdff}.} in plain language can be expressed as follows:\\
\begin{description}
\item[Theorem~\ref{gbearacdf43s0}] {\em Consider two bounded regions~$\Omega_1\subset\Omega_2$ of~$\R^n$
and a stationary prey uniformly distributed in~$\Omega_1$. Then, a stationary forager uniformly distributed in~$\Omega_2$
attains a higher efficiency than any diffusive forager which is initially uniformly distributed in~$\Omega_2$.}\\
\item[Theorem~\ref{kkjyyeffcdf5128}] {\em Consider a bounded region~$\Omega\subset\R^n$
and the case of a stationary prey uniformly distributed in a dilation~$\Omega_r$ by a factor~$r$ of~$\Omega$, combined with
a diffusive forager starting with a uniform distribution in~$\Omega_r$.
Then, large dilations and very long searching times are optimized by dispersals close to~$s=1$
(the case of small dilations requiring a finer discussion in terms of dimension).}\\
\item[Theorem~\ref{0ojrt9-43hgoiuf6hb7v65yTHMSMD}] {\em
If the Fourier transform of the distribution of a stationary prey is supported in a suitable ball, then this efficiency functional is monotone increasing in~$s$
and therefore the optimal strategy is given by the Gaussian dispersal~$s=1$.
If instead the Fourier transform of the distribution of a stationary prey is supported in the complement of a suitable ball, then the efficiency functional is monotone decreasing in~$s$
and therefore the optimal strategy is given by~$s=0$.}
\end{description}\medskip

In the forthcoming work~\cite{SPAN} we will also show that the framework proposed here is general enough to consider the case of
foragers with {\em finite lifetime} (i.e., of foragers who cannot survive after a certain time if they do not eat sufficient food).

The results presented in this section are just temporary, and somewhat imprecise, statements, aimed at
introducing problems, methodologies and solutions without relying on technical lingo
(also, Theorem~\ref{kkjyyeffcdf5128} will be further specified according to quantitative parameters which
shed light on the difference between the one-dimensional and the higher-dimensional cases).
Indeed, we will give a more precise statement of these results in Section~\ref{forma}
and compare them with the corresponding biological scenarios. Before that, we will analyze some limitation and criticalities of the setting into consideration in Section~\ref{DIAsdff}.

\section{Criticalities and limitations of the L\'evy flight foraging hypothesis}\label{DIAsdff}

Scientific theories are always constrained by the extent of existing knowledge and of the available experimental observability.
Biological theories often face the additional difficulty of dealing with very complex and multifaceted phenomena, in which many different components diversely contribute and influence each other. The vast interactions of the biological world are, on numerous occasions, a severe obstruction to the mathematical analysis of a given phenomenon, since, due to the technical difficulties involved, it would be more desirable to ``separate'' different scenarios, ``isolate'' different contributions and ``weigh'' them differently. However, in doing so, it is virtually impossible to avoid replacing the ``real world'' with a somewhat simplified model. This model, however, is not merely an abstract artifact and, every so often, is helpful towards a deeper understanding of the original phenomenon.\medskip

The L\'evy flight foraging hypothesis is not an exception in terms, on the one hand, of the creative stimulus towards the advance of knowledge
and, on the other hand, of the criticalities and limitations
intrinsic in its modelization and analysis. Actually, as it always happens to important scientific theories, the L\'evy flight foraging hypothesis is (and must be) constantly challenged, improved and revised.
In all this, in our humble opinion, it would be scientifically erroneous to simplistically
deem theories as ``right'' or ``wrong'' (for the same reason for which
Einstein's general relativity does not make Newtonian mechanics wrong), instead the main resource of science consists precisely in having different visions to compare, different theories to confront, 
different scenarios to which different approaches fit better, in an harmonious framework in which these diversities complement and complete each other.

In particular, it is often difficult to discern different causes which may contribute to the formation of L\'evy patterns: for example, in alternative to optimal searching strategies to detect scarce, randomly distributed resources, an alternative hypothesis has been also explored in~\cite{BOYE},
for which this pattern may result from the interaction of the foragers with a particular distribution of resources.

The literature is also full of concrete examples of theories and explanations
which involved the L\'evy flight foraging hypothesis
which have been intensively debated and modified over time. For instance,
\cite{ILP} monitored five wandering albatrosses on Bird Island by attaching
a salt-water immersion logger to one of their legs, with the aim of
recording flight durations (i.e., time intervals between landing on the ocean)
from the evidence of the logger being dry. In doing so, it
was assumed that birds landed on the water solely to forage
and the data collected seemed to strongly imply that
the albatrosses were performing
L\'evy flights. This conclusion was challenged in~\cite{REVIS}, due to a new set of high-resolution data,
which instead suggested the absence of L\'evy flights and 
that the extremely long flights, essential for demonstrating the L\'evy flight behavior, were actually spurious,
due to the fact that birds sitting on a nest would maintain their logger dry, without flying. 
See also~\cite{112-299535910-1182.1}, in which seventeen data sets from seven previous studies have been revisited and reconsidered, also stressing the importance of considering possible alternative hypotheses to L\'evy flights instead of taking them for granted (the final conclusion of~\cite{112-299535910-1182.1} being that ``L\'evy flights do not appear to have potential application as marine ecosystem indicators'').

Yet, that was not the end of the story. In~\cite{NUOKMDe0rotg}
an accurate GPS monitor of albatrosses detected both
L\'evy and Brownian movement patterns for individuals.
This is very telling, since it already highlights some of the specific criticalities of this research topic, such as the dependence of the results on the technology, on the instruments of measurements available at the time of the experiment, on the quantity and quality of data, and it emphasizes the difficulty in distinguishing between ``individual'' and ``collective'' behaviors.

In relation to this, also the path analysis of black bean aphids (small insects that feed on plants) carried out in~\cite{n9ut4hpibSDJMgrnc7n5}
evidenced a structural difference between individual and collective behaviors, in this case pointing out that
individuals may follow classical Brownian patterns, but, due to the variation at population level
(some individuals move much farther than others)
the collective motion displays superdiffusion structures similar to a L\'evy flight
(from the mathematical point of view, because individual movement lengths can be exponentially distributed, but with different parameters of this exponential distribution, so that the collective superposition of these exponential distributions
produce a power law of L\'evy type).

Furthermore,
it was proposed in~\cite{KAMSSNdkvse54NA} that
L\'evy movements accelerate mussels' pattern formations, but, confronting with a larger model set, it was later argued in~\cite{lmdcommskcvdCOMS}
that mussels' movements could be more accurately described by a composite Brownian walk.
This topic was reconsidered also in~\cite{0pqwojdlfeacshdnACCsq-lwd-fv}
bringing mussels' movements back into the family of L\'evy patterns
(though, as mentioned in~\cite{0pqwojdlfeacshdnACCsq-lwd-fv},
``for some of these other taxa, as with the mussels, it is possible that the tails of the step-length distributions are strictly not power-laws but are instead better represented by functions with multiple parameters'').

The comparison between L\'evy type diffusion and composite Brownian walks also emerged in~\cite{dec8f1c6-9001-3bee-9e37-330528089f10}
in which the alternation of extensive and intensive searching modes was related to a mixture of classical random walks, also stressing that an
 emergent movement patterns, such as a L\'evy distribution, should not be confused with the original process giving rise to the patterns, which can well be of classical type.
On this note, the role of the environmental feedback has been further analyzed in~\cite{BEAT},
addressing specifically the case in which the probability for the forager to detect the prey encountered during movement is low, further reinforcing the hypothesis in favor of a two-scale composite Brownian walk.

Also in relation to human mobility patterns, it is sometimes inferred from~\cite{FLLOW} that
human trajectories are best modeled by L\'evy distribution (roughly speaking,
given that money is carried by individuals, one could argue that the L\'evy type dispersal of
banknotes pointed out in~\cite{FLLOW} is a sufficiently accurate proxy for human movement). 
However, this conclusion was challenged in~\cite{GONZHUM}, on the basis that dollar bills follow the trajectory of their current owner
(thus exhibiting a diffusion phenomena) while individuals tend to return to a few highly frequented locations, such as home or work.
See also~\cite{K2345tlages2018}
for several examples and counterexamples of experimental data either confirming or refuting 
the L\'evy flight foraging hypothesis.

Another sector related to human activities in which L\'evy flights have been adopted is related to crime modeling. In this case, a proxy of the criminal motion is taken as the distance between criminals' homes and their targets as evidenced by solved crimes. Several sets of data indicate that burglars may undertake long trips toward attractive\footnote{It is in fact suggestive to compare the willingness of burglars to perform long excursions in view of more appealing targets with the ``ambush strategy'' for animal foragers looking for a high energy content prey, as it will be discussed on page~\pageref{AMdowfekgherikr0BUS}.} burglary sites and that experienced professional burglars are willing to travel farther than amateurs (see~\cite{16cf5b2d-4a8c-3bdc-aedb-9a006440ab8a, 800515fa-8f16-31e6-983d-80e0b8498330-98, fe9a402d-ee59-3159-bd62-493f1ef26fb8-11});
both agent-based models driven by the
fractional Laplacian operator
(see~\cite{MR3090652}) and mean-field continuum models with truncated L\'evy flights (see~\cite{MR3847185}) have been utilized to model these phenomena (and the two models are structurally different, since the continuum model ends up being more related to classical diffusion, rather than fractional diffusion).
\medskip

In addition, L\'evy flights are perhaps less ``universal'' than what one may expect and their efficiency highly dependent on small
modifications of the search conditions: for instance, it has been shown in~\cite{wpojlfJDlnf}
that the drift caused by currents or winds may let the L\'evy forager overshoot the target, favoring instead a classical Brownian forager,
or L\'evy foragers with exponents in the entire interval~$\left[\frac12,1\right]$.
\medskip

Other criticalities of the L\'evy flight foraging hypothesis include:
\begin{itemize}
\item Given the complicated genetic
interactions, there is no guarantee that natural selection alone
can always optimize a specific parameter (such as the L\'evy distribution exponent),
\item The classical L\'evy models do not take into account
the memory of the forager and the correlation between directions of successive movements,
\item The comparison between L\'evy and
Brownian movements often
does not consider other possible movement models,
\item The foraging strategy may depend on a large number
of parameters (density and mobility of preys, environmental conditions, etc.),
\item It is often difficult to distinguish
when animals are searching for food or
moving for some other reason (e.g. between known locations).
\end{itemize}
Given this and other difficulties, the L\'evy flight foraging hypothesis has received severe criticisms, see in particular~\cite{OkPY}.\medskip

Nonetheless,
we also recall that alternative (and somewhat reinforcing) versions of the L\'evy flight foraging hypothesis were proposed, see e.g.~\cite{REYNOLDS201559}, where the author considered the situation in which L\'evy patterns do not arise specifically to optimize the search of food but emerge spontaneously as byproducts of otherwise ``innocuous'' behaviors: if advantageous, these L\'evy patterns are maintained by the natural selection.\medskip

Several researchers also intensively discussed a technical criticality of L\'evy flights
related to the divergence of the corresponding mean square displacements.
Specifically, while the mean square displacement of the classical Brownian diffusion is linear in time
(hence the mean displacement grows as the square root of time),
the one associated with the L\'evy flights anomalous diffusion is infinite.
At first glance, this feature of the L\'evy flights seems problematic:
one may think that a model arising from
random walks with an infinite average square step size is a biological nonsense and that
no organisms can have movement patterns resembling L\'evy flights, since
in nature each step must be of some finite length
and no organism can escape to infinity in finite time. It is sometimes believed that the only way to circumvent this logical incongruence is to replace the notion of L\'evy flights with that of L\'evy ``walks''
(a continuous random walk typically executed at constant speed
with a coupled space-time
probability distributions penalizing excessively long steps). This has created a rather intense debate concerning
infinite versus finite propagation velocity, often considering L\'evy flights rather pathological.

However, in the words of George Box, ``all models are wrong, some are useful'' and the models based on  
L\'evy flights have the convenient advantage of producing diffusive operators that can be coherently and rigorously treated from a mathematical point of view. This is a strong indication that, maybe, after all,
the model is not only useful, but also not completely wrong.
Specifically, an infinite mean square displacement does not mean that all organisms escape to infinity straight away, it only means that while long steps are rare, they are not sufficiently rare to make their squared average dislocation finite. In other words, the divergence of the second moment
of the jump size distribution indicates that at least
few long steps are present, but the coherence of the model relies on the fact that most steps are much smaller, and, all in all, most, but not all, steps are ``quite short''.
See~\cite{Shlesinger1986, PHE} for further discussions about L\'evy flights and L\'evy walks.
\medskip

To confirm that L\'evy flights and objects with infinite mean square displacement do exist in nature, see e.g.~\cite{al2143terteioquwdjkfeng}, in which anomalous diffusion was observed, in a certain range of concentration and temperature, for a systems of
breakable micelles. In this framework, the fact that
the second moment of the jump size distribution diverges was interpreted as a byproduct of the fact that
short micelles diffuse more rapidly than long ones, in such a way that the effective diffusion constantly increases with time.

See also~\cite{1025t324}, where the observations of a flow in a rapidly rotating tank
in a suitably non-turbulent regime indicated the presence of L\'evy flight patterns. In this situation,
the divergent second moment for flight time resulted from the alternate behavior of the flow,
from sticking in vortices to long distance flies in the jets.
\medskip

{F}rom the historical point of view, it is interesting to recall that power law alternatives to more consolidated Gaussian models have always raised concerns and skepticism among more traditional scientists, the first occurrence of this dispute dating back at least to
the first half of the nineteenth century, see Mandelbrot's discussion in~\cite{MR1870013}, and
the specific comment according to which 
science is ``a  search  for insights  and  simple  invariances.    Power laws are consequences of scale invariance. They do not claim to provide a microscopically  precise  multiparameter  representation of everything.  What they aim at is a ``macroscopic'' description  of  reality''.

In fact, well before the L\'evy flight foraging hypothesis, issues related to fat tail distribution and infinite variance were highly debated among experts. Specifically, in the case just mentioned,
Mandelbrot had addressed the well accepted feature  of  financial prices to show an extremely ``anomalous''
variability by casting this phenomenon into the realm of L\'evy's stable distributions. This raised eyebrows,
since, according to~\cite{MR415945}, ``numerous earlier critics are sensible in
 their belief that infinite variance per se is a feature so undesirable that, in order to paper it over,
 the economist should welcome a finite-variance reformulation, even when [...] it is marred by otherwise undesirable features. [...] However, though the use of
 infinite variance distributions was called a radical departure, I don't believe it should continue to arouse emotion''. See also~\cite{43a28768-0b62-3fbb-bd95-562f24ac9127} for related discussions about the infinity of second moments in the approach by L\'evy and Mandelbrot.
\medskip

Let us also mention that the criticalities of the L\'evy flight foraging assumptions may also be part of
the broader discussion concerning the Marginal Value Theorem (often abbreviated as MVT), see~\cite{CHARNOV1976129, ojdlknfei43qonfgn8vBISkdPPOOjmfo}.
In this model, foragers perform a search within a ``patch'' and
preys are located in patches separated by areas with no resources, therefore
the search of food to meet the foragers' energetic needs is weighted by the energetically costly
travel between patches (this philosophy is somewhat also in common with our notion of efficiency functional).
In spite of the interest and the support for the Marginal Value Theorem,
some criticalities and limitations arose in this topic as well,
since it is often difficult to objectively measure payoff rates
and foraging times (especially in unfamiliar environments for the forager),
or to distinguish between the foraging activity and other activities (such as
avoiding predation risks or searching for mating opportunities).
These criticalities have motivated some researchers to state that
 ``the simplifying assumptions of the MVT introduce a systematic bias rather than just imprecision'', see~\cite{a00038109iuygfd12qewret}. See moreover Chapter~9 in~\cite{ojdlknfei43qonfgn8vBISkdPPOOjmfo} for further discussions about limitations and pitfalls in foraging models, also beyond the specific criticalities of the MVT or the L\'evy flight foraging hypothesis.
 So, once again, any research related to animal behavior has not to be taken literally, not as the conclusive answer, and always challenged and improved, to add yet a new bit to our modest knowledge\footnote{See e.g.~\cite{pkwfINAsGAndRiksdfduIS}, recalling ``the  need  for  intellectual  humility  or  being  transparent  about  and  owning  the  limitations  of  our  work.  Although  intellectual  humility  is  presented  as  a  widely  accepted  scientific  norm,  [...] current research practice does not incentivize intellectual humility. [...] We hope that by pushing each other to own the limitations of our work, we can reshape the incentives so that intellectual humility is rewarded and the credibility of science increases.''} of the universe.
\medskip

Let us stress that the interest of L\'evy distributions to optimize search strategies goes well beyond the realm of biology and, in fact, biologically inspired strategies have been adopted for several technological purposes. For example, 
L\'evy flights have found fruitful applications in the study of robot behavior and in the development of computationally effective methods to optimize the movements of robot swarms to cover a given area, find rare targets, or track objects (see~\cite{09540090412331314876, MR4065205, Duncan2022}). 
On a similar note, L\'evy flights have been utilized for optimization algorithms (see~\cite{1508807}) as well as wired networks
and mine detection (see~\cite{0015}). As for the biological case,
L\'evy flights should not be considered as a ``panacea'' and a fine tuning is often required (especially when the targets are clustered,
see~\cite{LFA}).
\medskip

Concerning the above mentioned criticalities, of course this paper does not purport to be the ultimate
answer in such a difficult and controversial topic. Instead, we aim at showing how these complications
actually arise even in very simple situations, thus underlining the importance of
tolerant, but rigorous, approaches, which sometimes cannot avoid precise, and maybe ``pedantic'', mathematical definitions and an explicit set of quantitative and unambiguous hypotheses.

We also remark that, with an abuse of notation, we speak about ``a prey'' or ``a forager'', while of course
solutions of the fractional heat equation under consideration must be considered, strictly speaking,
as ``distribution of preys'' or ``distribution of foragers''. With this respect, one of the advantages
of the approach using a fractional heat equation is that the distinction between individual and
collective behaviors is somewhat implicitly overcome in the model, since the equation is itself an averaged expression and can only produce ``aggregate'' results (with this respect, strictly speaking, even initial conditions corresponding to Dirac's Delta Function should not really be thought as an ``individual'' located at a specific point, but rather as a concentrate distribution of individuals -- but it is common to surf over
this subtle distinction and we will often adopt a simple language to describe the results, relegating the technical part to the mathematical theorems).

Another advantage of the fractional heat equation is that it somehow reduces the number of parameters involved in the analysis. This can also be seen as a limitation, since on many occurrences it can be desirable to include specific biological parameters (such as the sight of the animal, the previous knowledge about the environment, the interaction with other species, the fact that a search pattern could be terminated whenever a prey is found, possible spatial or temporal cutoffs, etc.). We do not aim at neglecting the importance of these parameters, which are oftentimes useful to obtain good fit with experiments, but we wanted to analyze the ``simplest possible'' model
which can already provide some useful information on a difficult and controversial topic.
\medskip

It is maybe worth completing this section by stressing that the theory of L\'evy flights clearly shows
the remarkable unity of different disciplines and goes beyond the rigid barriers which too often artificially separate ``theory'' and ``application''. As a confirmation of this fact we recall the fact mentioned in~\cite{Shlesinger1986}, according to which the original motivation for Paul L\'evy in the 1920's to introduce this kind of distributions was, in a sense, rather theoretical (though obviously motivated by very concrete applications):
at that time the vouge seemed to be that of proving the Central Limit Theorem under the most general possible assumptions,
and the work by L\'evy intended to provide a limitation for such generalizations by looking at random variables with infinite second moments. This approach would have eventually led to a broad and intense research activity involving mathematicians, physicists, biologists, engineers, computer scientists, etc., addressing the topics of anomalous diffusion, fractals, and fractional equations, from different viewpoints and with complementary techniques. And this, after all, is the beauty of science.

\section{Formal statement of the main results}\label{forma}

The mathematical framework that we consider takes into account a distribution of foragers, denoted by~$u(t,x)$, and of preys, denoted by~$p(t,x)$. Here, $x\in\R^n$ denotes the space variable and~$t\in[0,+\infty)$ denotes the time variable.

The efficiency functional accounting for the random encounters of foragers and preys takes the form
\begin{equation}
\label{EFFIFU9z123}
\frac1T\,\int_0^T \int_{\R^n} u(t,x)\,p(t,x)\,dx\,dt.
\end{equation}
The initial densities of foragers and preys are normalized to be probability distributions, i.e.
$$ \int_{\R^n} u(0,x)\,dx=\int_{\R^n} p(0,x)\,dx=1.$$
However, we will not necessarily assume here that these densities are nonnegative (hence, the notion of probability measure has to be intended as a ``signed'' measure\footnote{While
in the foraging problem densities of individuals are naturally nonnegative, our framework is general enough to comprise the situation in which the sign of the quantities involved is not specified, and this may find applications in other circumstances in which ``negative densities'' may indicate ``debts'' or ``gaps''.
{F}rom the technical point of view, this is useful for us when we consider Fourier transforms, see Figure~\ref{nijufLO-06557FTHNSojdwm}.} with unit total mass).

We will focus on the case of {\em stationary prey}~$p(t,x)=p(x)$, i.e.~$p$ independent of time. As for the forager, we will consider two main scenarios.
The first scenario is that of a {\em stationary forager} described by~$u(t,x)=u(x)$. The second consists of a {\em diffusive forager} evolving according to the fractional heat equation
\begin{equation}\label{WWD103023j131672452} \partial_t u=-(-\Delta)^{s} u.\end{equation}
In this setting, for all~$s\in(0,1)$, the operator~$-(-\Delta)^s$ is the {\em fractional Laplacian} defined as
\begin{equation}\label{WWD103023j13167245} -(-\Delta)^s \phi(x):=c_{n,s}\,\int_{\R^n} \frac{\phi(x+y)+\phi(x-y)-2\phi(x)}{|y|^{n+2s}}\,dy,\end{equation}
where~$c_{n,s}$ is a suitable\footnote{The role of this normalization constant is simply to allow a ``nice'' representation
of the fractional Laplacian operator in terms of the Laplace transform. Specifically, with this normalization,
the Fourier transform of~$(-\Delta)^s u(x)$ is simply~$(2\pi |\xi|)^{2s}\,\widehat u(\xi)$, where~$\widehat u$
is the Fourier transform of~$u$, namely
$$\widehat u(\xi)=\int_{\R^n} u(x)\,e^{2\pi ix\cdot\xi}\,dx.$$} normalization constant (specifically,
$c_{n,s}=\frac {2^{2s-1}\Gamma \left(\frac{n}2+s\right)}{\pi ^{n/2}|\Gamma (-s)|}$, where~$\Gamma$ is Euler's Gamma function).
An interesting role of this constant is to allow one to account for the cases~$s=0$ and~$s=1$
as limit situations in~\eqref{WWD103023j13167245} and thus consider~\eqref{WWD103023j131672452} for all~$s\in[0,1]$ in a unified notation. In particular, the case~$s=1$ corresponds to a classical Laplace operator on the right-hand side of~\eqref{WWD103023j131672452}, hence,
\eqref{WWD103023j131672452} for~$s=1$ reduces to the classical heat equation, generated by the standard random walk
(instead, for~$s\in(0,1)$, the infinitesimal generator of~\eqref{WWD103023j131672452} is a L\'evy flight). See also~\cite{MR2707618, MR3967804} for a basic introduction to the fractional Laplacian.

When~$s=0$,
equation~\eqref{WWD103023j131672452} reduces to~$\partial_t u=- u$.
In particular, for a forager initially concentrated at a given point~$y\in\R^n$, the case~$s=0$ produces
\begin{equation*}
\begin{dcases}
\partial_t u(t,x)=- u(t,x)\quad\mbox{for all}\quad (t,x)\in (0,+\infty)\times\mathbb{R}^n , \\
u(0,x)=\delta(y-x),
\end{dcases}
\end{equation*}
where~$\delta$ stands for Dirac's Delta Function. In this setting, the formal solution takes the form \begin{equation}\label{0ojlfe9375777584hgXojfg}
e^{-t}\delta(y-x).\end{equation}
This solution shows an interesting feature of the~$s=0$ case, which is perhaps not completely intuitive:
indeed~\eqref{0ojlfe9375777584hgXojfg} exhibits both a
lack of regularity (the solution is not a smooth function when~$t>0$, since it still presents a Dirac's Delta Function) and loss of mass (the total mass at time~$t$ being equal to~$e^{-t}$, showing that the mass is not preserved in time and suggesting that part of the mass is ``instantaneously'' pushed far away). These mathematical characteristics can be interpreted \label{AMdowfekgherikr0BUS}
as an ``ambush strategy'', in which a forager can remain still and suddenly perform a long jump, see Remark~2.10 in~\cite{ULTI}.
These ambush patterns have also been observed in nature, especially when the prey possesses high energy content,
see~\cite{PHE}, therefore they can also be related to ``high risk/high gain'' strategies.
Namely, this approach of
``risking it all for a big reward'', related to sudden ``surprise attacks'' by the forager striving for a highly advantageous target, is the counterpart of a small L\'evy exponent anomalous diffusion in which the low regularizing effect of the parabolic operator allows both a significant forager density to remain at the initial point and a notable part of this density to drift away.
\medskip

A natural question in foraging strategies is whether stationary foragers
would better fit stationary preys with uniform distributions sharing the same region.
We give a rigorous positive answer to this problem, according to the following result:

\begin{thm}\label{gbearacdf43s0}
Consider a stationary prey uniformly distributed in~$\Omega_1\subset\R^n$, i.e.
$$p(x)=\frac{\chi_{\Omega_1}(x)}{\left|\Omega_1\right|},$$
and suppose that~$\Omega_1\subset \Omega_2$.

Then, the efficiency functional associated with a stationary
forager uniformly distributed in~$\Omega_2$, i.e.
\begin{equation}\label{thesta}u(x)=\frac{\chi_{\Omega_2}(x)}{\left| \Omega_2\right|},\end{equation} is larger than that associated with a diffusive forager satisfying 
\begin{equation}\label{the}
\begin{cases}
\partial_t u(t,x)=-(-\Delta)^s u(t,x)\quad\mbox{for all}\quad (t,x)\in (0,+\infty)\times \R^n,\\
u(0,x)=\frac{\chi_{\Omega_2}(x)}{\left|\Omega_2 \right|},
\end{cases}
\end{equation}for any value of~$s\in[0,1]$.
\end{thm}

Here above and in the rest of this paper, we are adopting the standard notation for the indicator function of a set, namely
$$\chi_A(x):=\begin{cases}1&{\mbox{ if $x\in A$,}}\\0&{\mbox{otherwise.}}\end{cases}$$

We stress that Theorem~\ref{gbearacdf43s0} must not be taken literally in a biological context,
since one does not expect\footnote{There are however many examples of sit-and-wait predators in nature.
Some, such as the striated frogfish and the orchid mantis, use camouflage to attract prey, others, such as web-spinning spiders, build traps, others, such as the chameleons, use their abilities and special biological features (e.g., an exceptional tongue) to accomplish their goals. See e.g. Chapter~5 in~\cite{curio2012ethology} for further details.}
passive foragers waiting for a stationary prey to arrive.
The fact is that Theorem~\ref{gbearacdf43s0} describes an optimality
``at the group level'', dealing with distributions of foragers and preys,
not at individual level. The disparity of these approaches sits in the fact that
selection pressures seem not to operate at the group (i.e., population) level but rather
at the individual (i.e., genetic) level, making ``dynamic'' strategies more advantageous for single individuals.
Besides, as mentioned in Section~\ref{DIAsdff}, the search for food is not the unique pressure for an animal and other needs, such as avoiding predation risks or searching for mating opportunities, can favor movement and penalize stationary strategies. 
Also, Theorem~\ref{gbearacdf43s0} assumes the locations to be
revisitable and the prey to be instantaneously replaceable by a new prey, which is certainly an idealized situation (the scarcity of the remaining food after depleting the resource being otherwise an incentive for the forager to leave the patch).
In any case, results such as Theorem~\ref{gbearacdf43s0} only focus on the time of the forager dedicated to the search of food, not on the entire lifespan of the forager.
\medskip

In spite of these limitations, Theorem~\ref{gbearacdf43s0} provides a useful ``sanity check'' to attest the coherence of the model.
Also,
Theorem~\ref{gbearacdf43s0} raises the natural question of what the best diffusive forager is
(this is also a relevant biological question, since, as we have already pointed out, some biological foragers cannot be stationary and therefore cannot reach the optimal strategy of Theorem~\ref{gbearacdf43s0}).
In the next result we prove that the best diffusive search strategy (in terms of the most appropriate choice of~$s$) heavily depends on the size of the domain.

Notice indeed that the analysis of optimal foraging strategies in the class of fractional dispersal remains interesting and possibly nontrivial even when the forager and the prey share the same initial distribution, since,
even in the case of a stationary prey, the forager cannot be stationary and one has still
to detect the optimal exponent~$s$ for this type of search.

In this framework, one could expect, on the one hand, the exponent~$s=0$ to provide some kind of optimality, since, essentially, the distribution of the forager would remain qualitatively unchanged,
in view of~\eqref{0ojlfe9375777584hgXojfg},
yet some mass of the foragers is instantaneously delivered towards infinity.
On the other hand, one could also expect~$s=1$ to give some chance of optimality, since this would somewhat correspond to short Gaussian excursions. These considerations suggest that the optimality problem of stationary preys and fractional dispersive foragers with the same initial distributions is interesting and can reserve some surprises.

To clarify our result, given a bounded region~$\Omega\subset\R^n$,
we consider a stationary prey uniformly distributed in~$\Omega$, with corresponding distribution
$$p_\Omega(x):=\frac{\chi_{\Omega}(x)}{\left|\Omega\right|},$$
and a diffusive forager whose initial configuration is also uniformly distributed in~$\Omega$.
More explicitly, given~$s\in[0,1]$, the distribution of such a forager corresponds to the solution of
\begin{equation}\label{16BIS}
\begin{cases}
\partial_t u(t,x)=-(-\Delta)^s u(t,x)\quad\mbox{for all}\quad (t,x)\in (0,+\infty)\times \R^n,\\
u(0,x)=\frac{\chi_{\Omega}(x)}{\left|\Omega \right|},
\end{cases}
\end{equation}
and we denote by~$u_{s,\Omega}$ this solution. Thus, recalling~\eqref{EFFIFU9z123}, given
a time duration~$T>0$,
we consider the associated efficiency functional
\begin{equation}\label{16BIS-02ojf3egn} {\mathcal{J}}^{\Omega}(s,T):=
\frac1T\,\int_0^T \int_{\R^n} u_{s,\Omega}(t,x)\,p_\Omega(x)\,dx\,dt.
\end{equation}
Given~$r>0$, we also consider the dilation by a factor~$r$ of the region~$\Omega$, namely
$$\Omega_r:=\big\lbrace rx {\mbox{ s.t. }}x\in \Omega \big\rbrace.$$

Our result states that for each region~$\Omega$ we can find some time span $T\in (0,+\infty)$ of the search and a dilation~$\Omega_r$  for some large $r\in (1,+\infty)$, such that ${\mathcal{J}}^{\Omega_r}(\cdot,T)$ is strictly increasing. The biological intuition for this result is that short range dispersal strategies maintain
the foragers close to the prey, with a possible error localized near the boundary of the region, and a large dilation
makes this boundary region more negligible with respect to the bulk of the domain: consequently,
for large regions, in which the boundary effect becomes negligible with respect to the core of the domain, the classical Brownian motion becomes  optimal.

Moreover, we show that by considering instead  a small dilation $\Omega_r$, for some $r\in (0,1)$, then~${\mathcal{J}}^{\Omega_r}(\cdot,T)$ is strictly decreasing.
The biological intuition for this being that, if the distribution of prey occurs in a very small region, then any dispersal strategy would move the forager away from the prey, making it more convenient for the seeker to
adopt an extreme strategy which maintains it, as much as possible, at the initial position, possibly at the risk of being quickly
pushed towards infinity, since in any case small displacement would be sufficient to miss the prey.

The precise result needs to take into account the time span of the search and goes as follows:

\begin{thm}\label{kkjyyeffcdf5128}
Let $\Omega\subset \R^n$ be bounded, measurable and strictly convex. Then, there exists some~$r_{\Omega}\in (1,+\infty)$ such that
\begin{equation}\label{hnbg512dscxpae}
{\mathcal{J}}^{\Omega_r}(\cdot,T)\quad\mbox{is strictly increasing in $s\in (0,1)$,}
\end{equation} 
for each $r\in (r_{\Omega},+\infty)$ and  $T\in [r^2,+\infty)$.

Furthermore,  for each $T\in (0,+\infty)$ there exists some $r_{\Omega,T}\in (0,1)$ such that
\begin{equation}\label{gnapegcnr30d}
{\mathcal{J}}^{\Omega_r}(\cdot,T)\quad\mbox{is strictly decreasing in $s\in \left(0,\frac{n}{2}\right]\cap(0,1)$,}
\end{equation} 
for each $r\in (0,r_{\Omega,T})$. 

Also, if $n=1$, for each $\sigma\in \left(\frac{1}{2},1\right)$ there exist some $T_{\sigma}\in (1,+\infty)$ and $\tilde{r}_{\Omega}\in (0,1)$ such that 
\begin{equation}\label{fverd..34}
{\mathcal{J}}^{\Omega_r}(\cdot,T)\quad\mbox{is strictly increasing in $s\in \left(\sigma,1\right)$,}
\end{equation} 
for each $T\in (T_{\sigma},+\infty)$ and $r\in (0,\tilde{r}_{\Omega})$.
\end{thm}

Some comments about Theorem~\ref{kkjyyeffcdf5128} are in order. First of all, the claim in~\eqref{hnbg512dscxpae} holds true in every dimension and states, roughly speaking,
that, in case of large domains and very long searching times, the forager strategy is optimized by the Gaussian dispersal: this is consistent with the idea that very long times combined with
L\'evy flights at some point would make the forager jump out of the domain.

The claim in~\eqref{gnapegcnr30d} also holds true in every dimension, but its interpretation
varies according to the cases~$n=1$ and~$n\ge2$. When~$n\ge2$, the claim in~\eqref{gnapegcnr30d} boils down to
$$ {\mathcal{J}}^{\Omega_r}(\cdot,T)\quad\mbox{is strictly decreasing in $s\in(0,1)$,}$$
for each $r\in (0,r_{\Omega,T})$, saying that given any prescribed searching time, if the domain is sufficiently small, then the optimal strategy is~$s=0$. This is in agreement with the idea that the exponent~$s=0$ tends to keep a significant portion of the mass close to the original position.

However, the exponent~$s=0$ also tends to send a significant portion of the mass towards infinity, therefore these types of statements do require a careful analysis, and indeed the situation in dimension~$1$ is different. Indeed, when~$n=1$ the claim in~\eqref{gnapegcnr30d} becomes
\begin{equation*}
{\mathcal{J}}^{\Omega_r}(\cdot,T)\quad\mbox{is strictly decreasing in $s\in \left(0,\frac{1}{2}\right]$,}
\end{equation*} 
for each $r\in (0,r_{\Omega,T})$, giving that~$s=0$ outperforms every other strategy~$s\in\left(0,\frac12\right]$, leaving however the possibilities that other strategies~$s\in\left(\frac12,1\right]$
may be optimizers (or at least local optimizers). The role played by~$s=\frac12$ in this case is possibly related to the transient/recurrent phenomena of the L\'evy flights (see Section~2 in~\cite{MR4294710}). Specifically, when~$n\ge2$, and when~$n=1$ and~$s\in\left(0,\frac12\right)$, the L\'evy flight drifts away from the initial location, without coming back: in this setting, the foragers' mass sent out of the domain is not expected to return
and therefore keeping a portion of the mass at the original location becomes particularly efficient (thus favoring the strategy~$s=0$).

Instead, when~$n=1$ and~$s\in\left[\frac12,1\right]$, part of the mass sent out will eventually come back arbitrarily close to the original location. This effect is of course important when the searching time is sufficiently long, and this is the content of~\eqref{fverd..34}, which implies that
long searching times and small domains may produce~$s=1$ as a local optimizer (possibly because
some preys is found by the forager after exiting and reentering the domain).

We also remark that~\eqref{fverd..34} gives that the claim in~\eqref{gnapegcnr30d} is sharp,
since the intersection with the interval~$\left(0,\frac{n}{2}\right]$ cannot be dropped
(hence, the one-dimensional case is significantly different from the others).\medskip

It is certainly worth investigating deeper the special case~$n=1$ to further understand the differences
with respect to the higher-dimensional cases. For example, in the one-dimensional case, due to the lack of monotonicity
in the efficiency functional,
one can also look at pessimizers for the foraging strategies. In this context,
it follows from a careful analysis of~\eqref{gnapegcnr30d} and~\eqref{fverd..34} that in the one-dimensional case the functional $\mathcal{J}^{\Omega_r}(\cdot,T)$ develops a global minimizer in a neighbourhood of $s=\frac{1}{2}$, for $r$ small enough and $T$ large enough. More precisely, we can state the following result.

\begin{cor}\label{kniredl998}
Let $n=1$ and $\Omega$ be some bounded interval. Then, there exist some $T_{\Omega}\in (1,+\infty)$ and $r_{\Omega}\in (0,1)$  such that for each $r\in (0,r_{\Omega})$, $T\in (T_{\Omega},+\infty)$ and some $\sigma_T\in \left[\frac{1}{2},1\right)$
\begin{equation}\label{hbre4352618}
\mathcal{J}^{\Omega_r}(\cdot,T)\,\mbox{ has a global minimizer in the interval }\, s\in \left[\frac{1}{2},\sigma_{T}\right].
\end{equation}

Furthermore, it holds that 
\begin{equation}\label{gllpldmtr}
\lim_{T\to +\infty} \sigma_{T}=\frac{1}{2}.
\end{equation} 
\end{cor}

In the one-dimensional case, for each $s\in \left(0,\frac{1}{2}\right)$, we will consider the functional\footnote{The right-hand side of~\eqref{knrec53vdo97t} diverges when~$s\ge\frac12$ and this is the reason for us to restrict our attention to
the range~$s\in \left(0,\frac{1}{2}\right)$ in this case.}
\begin{equation}\label{knrec53vdo97t}
\mathcal{J}_\infty^{\Omega}(r,s):= r^{2s-1}\int_{\R}\frac{\left|\widehat{\chi}_{\Omega}(\xi) \right|^2}{(2\pi \left|\xi \right|)^{2s}}\,d\xi.
\end{equation}
As stated in \eqref{lmoetcbt5432} here below, this functional corresponds to the limit for the time span $T$ diverging to infinity of the functional $T\mathcal{J}$, therefore it is interesting to consider this functional as a proxy for the {\em renormalized efficiency
of a very long search}.

We establish that $\mathcal{J}_\infty^{\Omega}(r,\cdot)$ is strictly increasing for each dilation of parameter $r$ large enough. Furthermore, we also prove that for every subset $(0,\sigma)\subset \left(0,\frac{1}{2}\right)$ there exists some $r_{\sigma,\Omega} \in (0,1)$ such that the functional $\mathcal{J}_\infty^{\Omega}(r,\cdot)$ is strictly decreasing in $(0,\sigma)$ for each $r\in (0,r_{\sigma,\Omega})$. To conclude, we also show that if $\sigma\nearrow \frac{1}{2}$, then the parameter $r_{\sigma,\Omega}$ converges to zero.

\begin{cor}\label{pfff}
Let $\Omega\subset \R$ be some bounded interval and $\mathcal{J}_\infty^{\Omega}(r,\cdot)$ be given as in equation~\eqref{knrec53vdo97t}. Then, for each $s\in \left(0,\frac{1}{2}\right)$ and $r\in (0,+\infty)$ it holds that
\begin{equation}\label{lmoetcbt5432}
\lim_{T\to +\infty} T \mathcal{J}^{\Omega_r}(s,T)=\mathcal{J}_\infty^{\Omega}(r,s).
\end{equation} 
Also, there exists some $r_\Omega\in (0,+\infty)$ such that, for each $r\in (r_{\Omega},+\infty)$,
\begin{equation}\label{hbcyevco87tb}
\mathcal{J}_\infty^{\Omega}(r,\cdot)\,\mbox{ is strictly increasing in the interval }\, s\in \left(0,\frac{1}{2}\right).
\end{equation}

Furthermore, for each $\sigma \in \left(0,\frac{1}{2}\right)$, there exists $r_{\sigma,\Omega}\in (0,+\infty)$,
with
\begin{equation}\label{ukybgiygiy}
\lim_{\sigma\nearrow \frac{1}{2}} r_{\sigma,\Omega} =0
\end{equation}
and such that,
\begin{equation}\label{uytvuytvuytvfruytf}\begin{split}&
{\mbox{for each $r\in (0,r_{\sigma,\Omega})$, the function }}\,
\mathcal{J}_\infty^{\Omega}(r,\cdot)\\ &\mbox{ is strictly decreasing in the interval}\;s\in \left(0,\sigma\right)\end{split}
\end{equation}
and
\begin{equation}\label{uytvuytvuytvfruytf2}\begin{split}&
{\mbox{for each $r\in (r_{\sigma,\Omega},+\infty)$, the function }}\,
\mathcal{J}_\infty^{\Omega}(r,\cdot)\\&\mbox{cannot be decreasing in the interval}\;s\in \left(0,\sigma\right).\end{split}
\end{equation}
\end{cor}

In a nutshell, the result in~\eqref{hbcyevco87tb} suggests that, for very long time asymptotics of a one-dimensional searching pattern, when the prey is uniformly distributed in a large interval,
the pessimizer occurs at~$s=0$ and the optimizer at~$s=\frac12$.

The situation for small intervals is instead different, since~\eqref{uytvuytvuytvfruytf} shows that
for any~$\sigma\in\left(0,\frac{1}{2}\right)$ the pessimizer cannot occur in the interval~$[0,\sigma]$,
provided that the domain is sufficiently small.

All in all, these results warn us about the complexity of the foraging problem, which needs to quantitatively and rigorously examine
different cases (depending on dimensionality, distribution of preys, etc.), since each scenario presents its own specificity.
\medskip

Some further comments about~\eqref{ukybgiygiy} and~\eqref{uytvuytvuytvfruytf2} are in order.
These claims state the sharpness of the choice of~$r_{\sigma,\Omega}$ in Corollary~\ref{pfff}.
Namely, by~\eqref{uytvuytvuytvfruytf2} one knows that~$r_{\sigma,\Omega}$ is taken ``as large as possible''
(of course, choosing a smaller value would still produce~\eqref{uytvuytvuytvfruytf}, but
this would violate~\eqref{uytvuytvuytvfruytf2}).

Interestingly, even this choice of~$r_{\sigma,\Omega}$ in its largest possible sense does not allow it to be bounded from zero
as~$\sigma\nearrow \frac{1}{2}$, as expressed in~\eqref{ukybgiygiy}, thus
showing the quantitative optimality of the claims in Corollary~\ref{pfff}.
\medskip

{F}rom a biological standpoint, it is also interesting to consider the case of a stationary prey which is not uniformly distributed and a diffusive forager that aims at optimizing its searching strategy. In this case, we consider a stationary distribution of prey~$p\in L^2(\R^n)$
and a diffusive forager with initial distribution equal to the prey. Thus, the distribution of forager
is a probability density satisfying
\begin{equation}\label{kicfd192u43rtygihbX}
\begin{cases}
\partial_t u(t,x)=-(-\Delta)^s u(t,x)\quad\mbox{for all}\quad (t,x)\in (0,+\infty)\times \R^n,\\
u(0,x)=p(x).
\end{cases}
\end{equation}
We will denote by~$u_p$ the solution of~\eqref{kicfd192u43rtygihbX}.
We recall that we are not assuming here that~$p\ge0$ (as already observed, from the biological point of view, we may consider negative values of~$p$
as denoting a lack of resources, or by regions containing poisonous food).
In light of~\eqref{EFFIFU9z123}, we consider in this setting
the efficiency functional of the form
\begin{equation}\label{234dpG25ictudwf23re}\int_{0}^T\int_{\R^n} u_p(t,x)\,p(x)\,dx\,dt.\end{equation}
For simplicity, we are dropping here the normalization obtained by dividing by~$T$ (compare~\eqref{16BIS-02ojf3egn} and~\eqref{234dpG25ictudwf23re}), since this normalization does not play a role when optimizing in the diffusion exponent~$s$.

Thus, we have:

\begin{thm}\label{0ojrt9-43hgoiuf6hb7v65yTHMSMD}
The efficiency functional~\eqref{234dpG25ictudwf23re} associated with a forager diffusing as in~\eqref{kicfd192u43rtygihbX} has the form
\begin{equation}\label{0ojrt9-43hgoiuf6hb7v65yc70bv65bvcnonyTDv8y4bv45ybv}
\int_{0}^T\int_{\R^n} e^{-(2\pi |\xi|)^{2s}t}\,|\widehat p(\xi)|^2
\,d\xi\,dt.\end{equation}

In particular, if the Fourier transform of~$p$ is supported in the ball of  radius $\frac{1}{2\pi}$, then this efficiency functional is monotone strictly increasing in~$s$
and therefore the optimal strategy is given by the Gaussian dispersal~$s=1$.

If instead the Fourier transform of~$p$ is supported in the complement of the ball of radius $\frac{1}{2\pi}$, then this efficiency functional is monotone strictly decreasing in~$s$
and therefore the optimal strategy is given by~$s=0$.
\end{thm}

\begin{center}
\begin{figure}[!ht]
\includegraphics[width=0.29\textwidth]{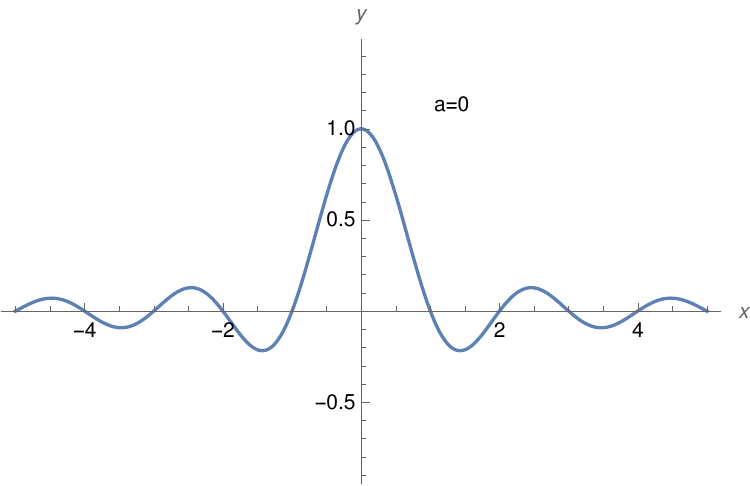}$\qquad$
\includegraphics[width=0.29\textwidth]{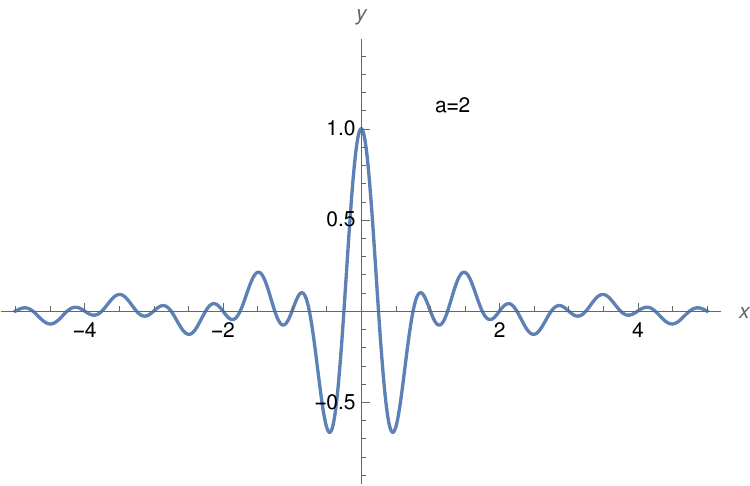}$\qquad$
\includegraphics[width=0.29\textwidth]{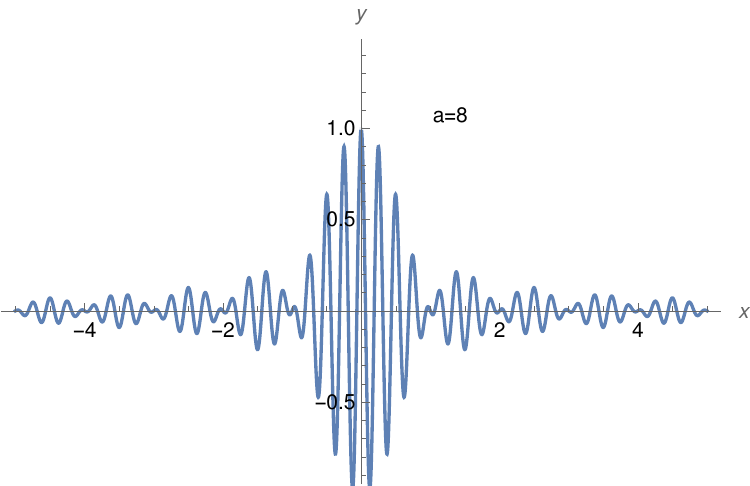}
\caption{\sl\footnotesize Plots of the inverse Fourier transform of~$\chi_{(-1/2,1/2)}(\xi-a)$, with~$a\in\{0,2,8\}$. Analytically, this corresponds to the plots of~$y=\frac{\cos(2\pi a x)\sin(\pi x)}{\pi x}$.}
        \label{nijufLO-06557FTHNSojdwm}
\end{figure}
\end{center}

Let us comment on the consistency of Theorem~\ref{0ojrt9-43hgoiuf6hb7v65yTHMSMD}
with the biological scenario. Roughly speaking, distributions whose Fourier transform are supported close to the origin
present ``mild oscillations'', while the ones whose Fourier transform has support reaching towards infinity
present ``wilder oscillations'' (see e.g. Figure~\ref{nijufLO-06557FTHNSojdwm}, for an explicit example of oscillations increased by shifting the support of the Fourier transform). Thus, in a sense, the monotonicity statement in Theorem~\ref{0ojrt9-43hgoiuf6hb7v65yTHMSMD} suggests that:
\begin{itemize}
\item on the one hand, when the distribution of preys has mild oscillations,
the classical random walk is better than any fractional diffusion (coherently with the intuition that
it is not convenient for the forager to leave its original location),
\item on the other hand, when the distribution of preys has wild oscillations,
the fractional parameter~$s=0$ is the optimizer (because, roughly speaking, it is more convenient for the
forager to try not to move at all, possibly at the risk of being sent to infinity straight away, since in any case small fluctuations
in a highly oscillatory framework
may end up in inconvenient regions for the prey distribution, recall the discussion below~\eqref{0ojlfe9375777584hgXojfg}).
\end{itemize}

\begin{center}
\begin{figure}[!ht]
\includegraphics[width=0.29\textwidth]{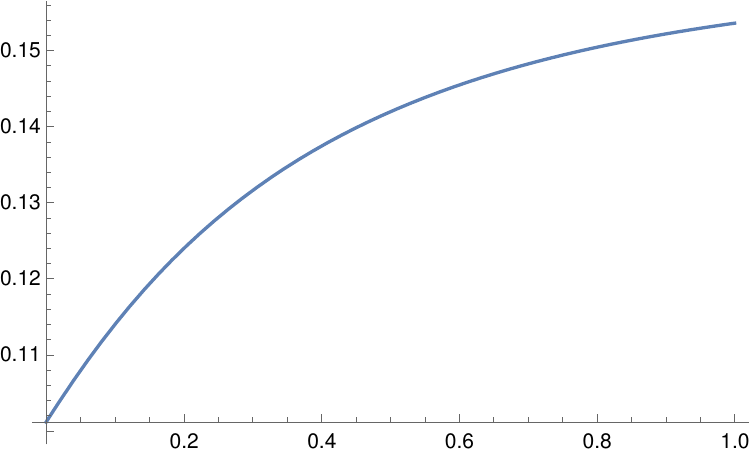}$\qquad$
\includegraphics[width=0.29\textwidth]{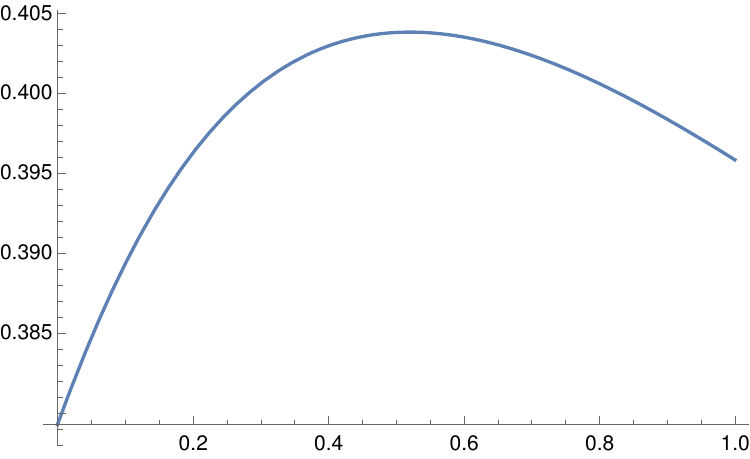}$\qquad$
\includegraphics[width=0.29\textwidth]{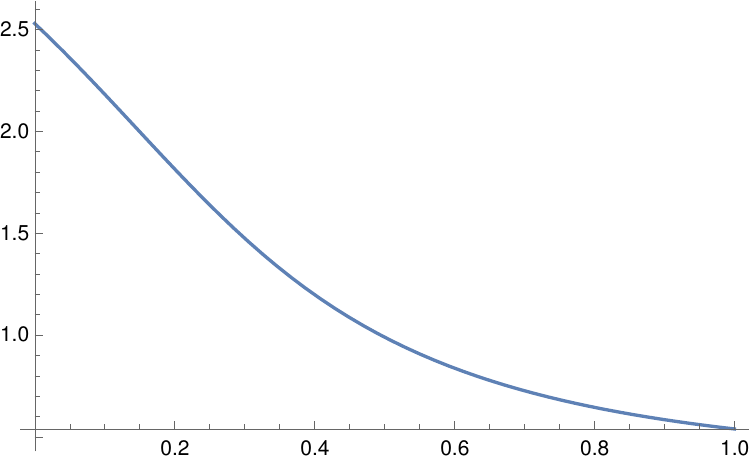}
\caption{\sl\footnotesize Plots of the efficiency functional in~\eqref{0ojrt9-43hgoiuf6hb7v65yc70bv65bvcnonyTDv8y4bv45ybv} when~$T=1$
and~$\widehat p=\chi_{(-a,a)}$, with~$a\in\{0.08,\,0.3,\,2\}$. Notice the change
in monotonicity in the fractional exponent~$s$.}
        \label{nijufLO-06512q3w4etryui99-57FTHNSojdwm2}
\end{figure}
\end{center}

The change of monotonicity detected in Theorem~\ref{0ojrt9-43hgoiuf6hb7v65yTHMSMD}
can be visualized in Figures~\ref{nijufLO-06512q3w4etryui99-57FTHNSojdwm2} and~\ref{nijufLO-06512q3w4etryui99-57FTHNS-000ojdw2233m}
(see also Figure~\ref{bhcaomcte45bc01} for a graphical sketch in higher dimension).\medskip

We can also interpret the results stated in Theorem~\ref{kkjyyeffcdf5128} in light of Theorem~\ref{0ojrt9-43hgoiuf6hb7v65yTHMSMD}, according, very roughly speaking, to a  prescription of the type
\begin{equation}\label{p0ewojfhnPOijsdfIDCACCA01M}\boxed{
\begin{gathered}
{\mbox{small domain}}\\ \Downarrow \\ {\mbox{considerable mass of the Fourier transform at infinity}}\\
\Downarrow\\ {\mbox{efficiency maximizer close to }}s=0,
\end{gathered}}\end{equation}
as well as
\begin{equation}\label{p0ewojfhnPOijsdfIDCACCA01M2}\boxed{
\begin{gathered}
{\mbox{large domain}}\\ \Downarrow \\ {\mbox{mass of the Fourier transform concentrated at the origin}}\\
\Downarrow\\ {\mbox{efficiency maximizer close to }}s=1.
\end{gathered}}\end{equation}

That is, if one considers, as a model case, a uniform distribution of prey in the interval~$(-a,a)$,
the corresponding Fourier transform takes the form~$\widehat p(\xi)=\frac{\sin(2\pi a\xi)}{\pi\xi}$.
Hence, one can renormalize the efficiency functional in~\eqref{0ojrt9-43hgoiuf6hb7v65yc70bv65bvcnonyTDv8y4bv45ybv} by dividing by~$a^2$ (since this does not change the optimization procedure in~$s$): in this way, a noticeable quantity for the efficiency functional becomes
\begin{equation}\label{nijufLO-134er06512q3w4etryui99-57FTHNS-000ojdw2233m83i4rA2tjhkgnT2} \frac{|\widehat p(\xi)|^2}{a^2}=\frac{\sin^2(2\pi a\xi)}{(\pi a\xi)^2}.\end{equation}
The advantage of this normalization is that the value at~$\xi=0$ of this function is always equal to~$4$, for every choice of~$a>0$. A sketch of such a function for different values of~$a$ is given in Figure~\ref{nijufLO-134er06512q3w4etryui99-57FTHNS-000ojdw2233m83i4rA2tjhkgnT}. 

{F}rom these diagrams it is apparent that for small values of~$a$ the function in~\eqref{nijufLO-134er06512q3w4etryui99-57FTHNS-000ojdw2233m83i4rA2tjhkgnT2} presents a significant portion of mass far from the origin (indeed, its limit as~$a\to0$ is identically equal to~$4$), and therefore,
by Theorem~\ref{0ojrt9-43hgoiuf6hb7v65yTHMSMD}, we would expect the efficiency functional to be maximized for values of the L\'evy exponent close to $s=0$, which is indeed proved to be true in Theorem~\ref{kkjyyeffcdf5128}. 
Note that these considerations are consistent with the heuristic presented in~\eqref{p0ewojfhnPOijsdfIDCACCA01M}.

Instead, for large values of~$a$ the function in~\eqref{nijufLO-134er06512q3w4etryui99-57FTHNS-000ojdw2233m83i4rA2tjhkgnT2} concentrates at the origin (indeed, its limit as~$a\to+\infty$ is the indicator function of the origin), and therefore,
by Theorem~\ref{0ojrt9-43hgoiuf6hb7v65yTHMSMD}, we would expect the efficiency functional to be maximized for values of the L\'evy exponent close to $s=1$, which is indeed proved to be true in Theorem~\ref{kkjyyeffcdf5128}. 
Note that these considerations are consistent with the heuristic presented in~\eqref{p0ewojfhnPOijsdfIDCACCA01M2}.

\begin{center}
\begin{figure}[!ht]
\includegraphics[width=0.29\textwidth]{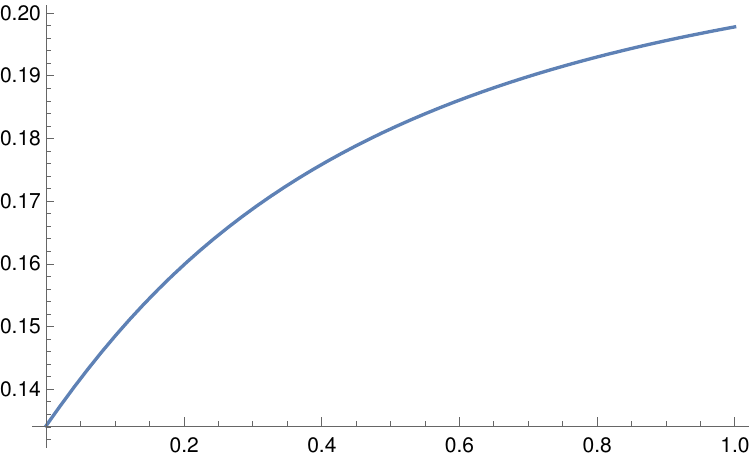}$\qquad$
\includegraphics[width=0.29\textwidth]{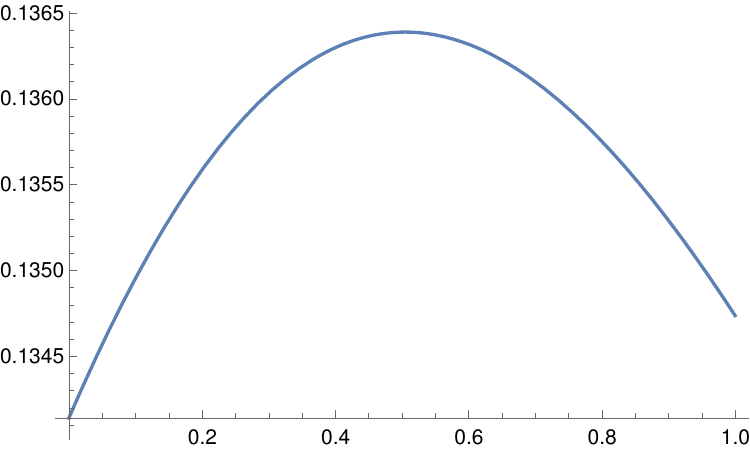}$\qquad$
\includegraphics[width=0.29\textwidth]{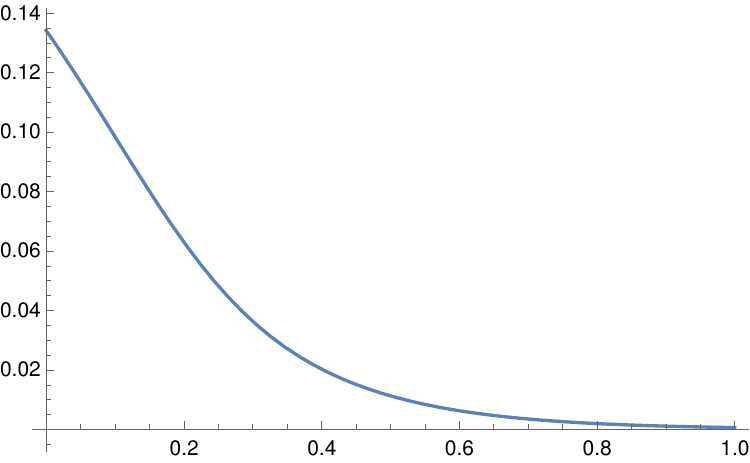}
\caption{\sl\footnotesize Plots of the efficiency functional in~\eqref{0ojrt9-43hgoiuf6hb7v65yc70bv65bvcnonyTDv8y4bv45ybv} when~$T=1$
and~$\widehat p=\chi_{\left(r-\frac1{3\pi},r+\frac1{3\pi}\right)}$, with~$r\in\{0,\,0.16,\,3\}$. Notice the change in monotonicity in the fractional exponent~$s$.}
        \label{nijufLO-06512q3w4etryui99-57FTHNS-000ojdw2233m}
\end{figure}
\end{center}

We stress however that the argument for large domains presents an important caveat:
namely, as stressed in Theorem~\ref{kkjyyeffcdf5128}, the optimizer for large domain is
indeed close to~$s=1$, but only provided that the search lasts long enough. In Figure~\ref{niliuer7g} we exhibit an example of how, for the same initial domain~$\Omega$, we can have a change of monotonicity for~${\mathcal{J}}^{\Omega}$, as we  consider different lifetimes $T\in (0,+\infty)$.

\begin{center}
\begin{figure}[!ht]
\includegraphics[width=0.29\textwidth]{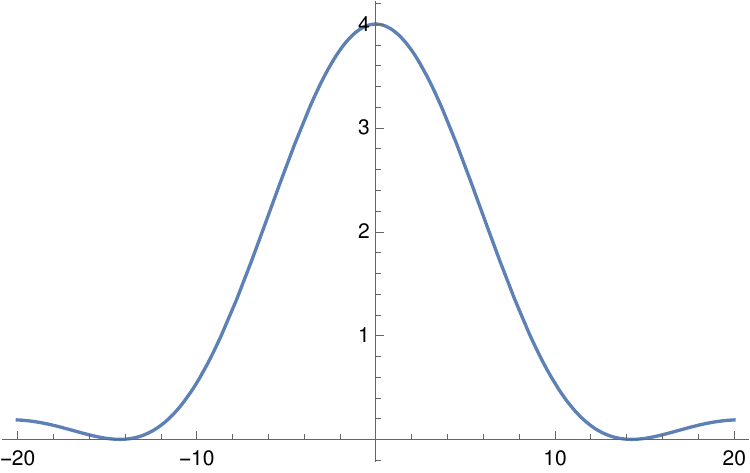}$\qquad$
\includegraphics[width=0.29\textwidth]{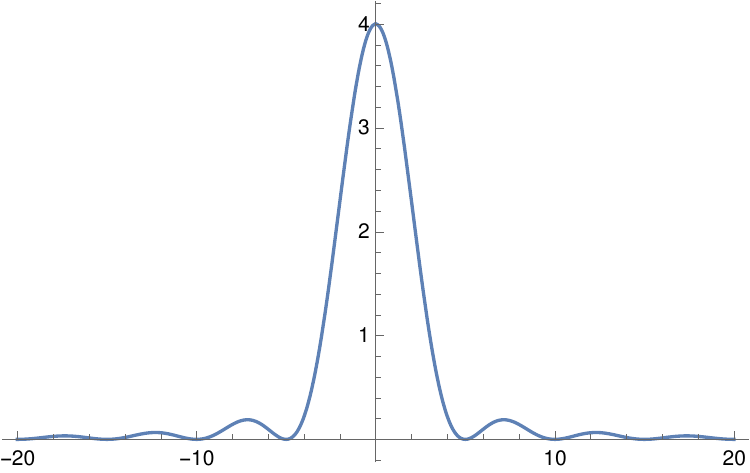}$\qquad$
\includegraphics[width=0.29\textwidth]{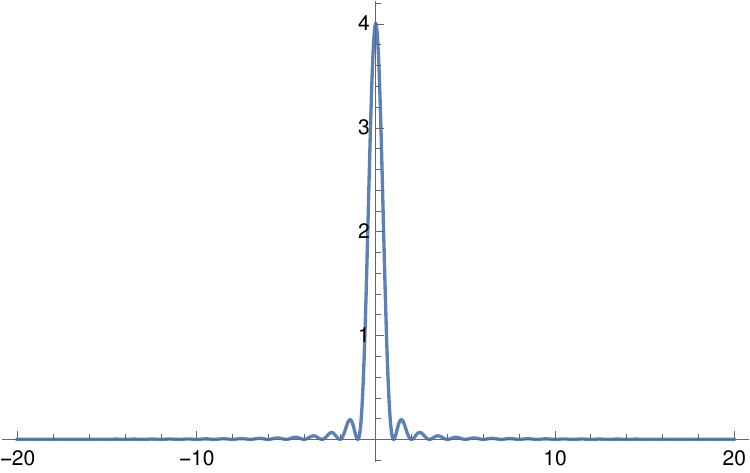}
\caption{\sl\footnotesize Plots of~$\frac{\sin^2(2\pi a\xi)}{(\pi a\xi)^2}$ with~$a\in\{0.035,\,0.1,\,0.5\}$.}
        \label{nijufLO-134er06512q3w4etryui99-57FTHNS-000ojdw2233m83i4rA2tjhkgnT}
\end{figure}
\end{center}

This is due to the fact that, for small times $T\in (0,+\infty)$, most of the mass of a forager diffusing according to~\eqref{16BIS} for some~$s$ close to zero is fixed at the initial location, with only a small part of it sent to infinity. In this case, even for a large domain, a strongly nonlocal approach seems more convenient, since most of the mass will be confined in $\Omega$.

Instead, when the time span $T\in (0,+\infty)$ is large enough, the monotonicity properties of ${\mathcal{J}}$ change. Indeed, as $s$ is close to zero, as a consequence of the large time span $T$, we have that a greater part of the forager mass is going to infinity, thus exiting the perimeter of the domain $\Omega$ where the prey is distributed. In this case, it is more convenient to move according to a strategy close to the Brownian motion. Indeed, this leads to short-medium jumps which allow to search the domain  while letting as little mass as possible out of $\Omega$.  \medskip

\begin{center}
\begin{figure}[!ht]
\includegraphics[width=0.40\textwidth]{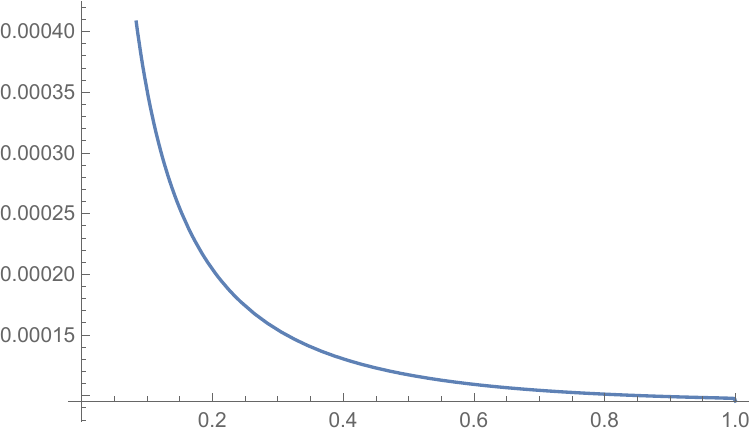}$\qquad$
\includegraphics[width=0.40\textwidth]{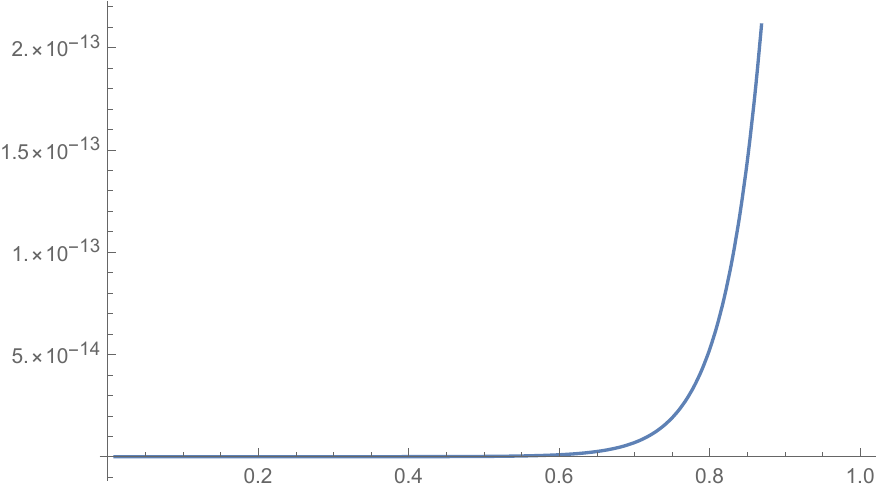}$\qquad$
\caption{\sl\footnotesize Plots of $ (0,1)\ni s\to {\mathcal{J}}^{\Omega}(s,T)$, with $\Omega=(-10^{4},10^4)$ and $T\in \left\lbrace 1,10^{8} \right\rbrace$.}
        \label{niliuer7g}
\end{figure}
\end{center}

It would be also interesting to consider more general cases of the above results, in which an active interaction between the forager and the prey takes place. For instance, one could consider the case of foragers with the ability of ``smelling'' preys, which could produce a drift, or chemotactic, term in the diffusive equation.
An interesting model for foragers with long range smell has been proposed in~\cite{Sanhedrai2019}, see also~\cite{PhysRevLett113238101}.
This model also accounts for the situation in which the forager
has a finite lifetime (that is, it can forage only for a finite time, since if it
does not eat any food for a given time interval, then it dies). This is a very interesting possibility
and we aim at exploring it further in future works~\cite{SPAN}.
\medskip

We now provide the proofs of the above results.

\section{Proof of the main results}\label{tgbepa-1}

Given~$y\in\R^n$ and~$s\in[0,1]$, we denote by~$G^s(t,x,y)$ the solution of the
fractional heat equation with initial condition given by
the Dirac's Delta Function at~$y$, namely~$G^s(t,x,y)$ solves
\begin{equation*}
\begin{dcases}
\partial_t u(t,x)=-(-\Delta)^s u(t,x)\quad\mbox{for all}\quad (t,x)\in (0,+\infty)\times\mathbb{R}^n , \\
u(0,x)=\delta(y-x).
\end{dcases}
\end{equation*}
Also, if~$u_p$ is the solution to~\eqref{kicfd192u43rtygihbX}, then we can write
\begin{equation}\label{dfejirv4b5uy45uityi}
u_p(t,x)=\int_{\R^n} {G}^s(t,x,y)\,p(y)\,dy.\end{equation}

We now start the proofs of the main results by checking~\eqref{0ojrt9-43hgoiuf6hb7v65yc70bv65bvcnonyTDv8y4bv45ybv},
so that we will be able to use it in the proofs of the main theorems.

\begin{proof}[Proof of~\eqref{0ojrt9-43hgoiuf6hb7v65yc70bv65bvcnonyTDv8y4bv45ybv}]
Taking the Fourier transforms in the spatial variables
of~\eqref{dfejirv4b5uy45uityi} we see that
$$ \widehat u_p(t,\xi)=\int_{\R^n} \widehat{G}^s(t,x,y)\,p(y)\,dy=
\int_{\R^n} e^{-(2\pi |\xi|)^{2s}t-2\pi iy\cdot\xi}\,p(y)\,dy.$$
Hence, by~\eqref{234dpG25ictudwf23re} and Plancherel's Theorem,
denoting with the bar the complex conjugation,
the efficiency functional can be written as
\begin{eqnarray*}&&
\int_{0}^T\int_{\R^n} u_p(t,x)\,\overline{p(x)}\,dx\,dt
=\int_{0}^T\int_{\R^n} \widehat u_p(t,x\xi)\,\overline{\widehat p(\xi)}\,d\xi\,dt\\&&\qquad=
\int_{0}^T\int_{\R^n}\int_{\R^n} e^{-(2\pi |\xi|)^{2s}t-2\pi iy\cdot\xi}\,p(y)\,\overline{\widehat p(\xi)}
\,dy\,d\xi\,dt
=\int_{0}^T\int_{\R^n} e^{-(2\pi |\xi|)^{2s}t}\,\widehat p(\xi)\,\overline{\widehat p(\xi)}
\,d\xi\,dt,
\end{eqnarray*}
giving~\eqref{0ojrt9-43hgoiuf6hb7v65yc70bv65bvcnonyTDv8y4bv45ybv}.
\end{proof}

Now we address the proofs of the main theorems.

\begin{proof}[Proof of Theorem~\ref{gbearacdf43s0}] In light of~\eqref{EFFIFU9z123}, the efficiency functional associated with a stationary forager as in~\eqref{thesta} is
\begin{equation}  
\begin{split}
&\frac{1}{T}\int_{0}^T\int_{\mathbb{R}^n} u(t,x)p(t,x)\,dx\,dt=\frac{1}{T}\int_{0}^T\int_{\mathbb{R}^n}\frac{\chi_{\Omega_2}(x)}{\left|\Omega_2\right|}\frac{\chi_{\Omega_1}(x)}{\left|\Omega_1\right|}\,dx\,dt=\frac{\left|\Omega_1\cap \Omega_2 \right|}{\left|\Omega_1 \right|\left|\Omega_2 \right|}. 
\end{split}
\end{equation}
In particular, if $\Omega_1\subset \Omega_2$ we have that this value coincides with $\frac{1}{\left| \Omega_2 \right|}$.

Recalling~\eqref{dfejirv4b5uy45uityi} and owing to the Maximum Principle for the fractional heat equation~\eqref{the}, we  have that
the efficiency functional in the fractional diffusion case is bounded from above by
\begin{equation*}
\begin{split}
\frac{1}{T \left|\Omega_1 \right|\left|\Omega_2 \right|}\int_{0}^T\int_{\Omega_1\times \Omega_2}G^s(t,x,y)\,dx\,dy\,dt\leq \frac{1}{T \left|\Omega_1\right| \left|\Omega_2 \right|}\int_{0}^T\int_{\Omega_1}\,dx\,dt=\frac{1}{\left|\Omega_2 \right|}.
\end{split}
\end{equation*}
This completes the proof.
\end{proof}


The material developed so far is sufficient to prove Theorem~\ref{kkjyyeffcdf5128}. For the reader's facility, this
technical proof is deferred to Appendix~\ref{APKmddu0oj4g034JLA}.

\begin{proof}[Proof of Corollary~\ref{kniredl998}]
We begin by recalling~\eqref{ooorefdvcte}. In this way,
if $r\in (0,\tilde{r}_\Omega)$ and $T\in (4,+\infty)$, we have that, for each $\sigma \in \left(\frac{1}{2},1 \right)$ and~$s\in (\sigma,1)$,  
\begin{equation*}
\partial_s\mathcal{J}^{\Omega_r}(s,T) \geq  \frac{1}{T}\left({C}\ln\left(\frac{T}{4}\right)T^{1-\frac{1}{2s}}-\frac{\left|\Omega\right|^2}{\pi (2\sigma-1)^2} \right).
\end{equation*}
Thus, since in this setting~$1-\frac{1}{2s}\ge0$, and accordingly~$T^{1-\frac{1}{2s}}\ge1$, we find that
\begin{equation}\label{SIGDE0}
\inf_{s\in(\sigma,1)}\partial_s \mathcal{J}^{\Omega_r}(s,T)\geq  \frac{1}{T}\left({C}\ln\left(\frac{T}{4}\right)-\frac{\left|\Omega\right|^2}{\pi (2\sigma-1)^2} \right).
\end{equation}

Now we define
\begin{equation}\label{SIGDE} \sigma_T:=\frac12+\frac{|\Omega|}{2\sqrt{C\pi\ln(T/4)}}.\end{equation}
We observe that~$\sigma_T\in\left(\frac12,1\right)$ for~$T$ large enough, and in fact
\begin{equation}\label{gllpldmtr-2}
\lim_{T\to+\infty}\sigma_T=\frac12.
\end{equation}

In view of~\eqref{SIGDE0} and~\eqref{SIGDE}, we find that
\begin{equation*}
\inf_{s\in(\sigma_T,1)}\partial_s \mathcal{J}^{\Omega_r}(s,T)\ge0,
\end{equation*}
whence the minimum of~${\mathcal{J}}^{\Omega_r}(\cdot,T)$ must occur for~$s\in[0,\sigma_T]$.

However, when~$r>0$ is small enough, by~\eqref{gnapegcnr30d} (applied here with~$n=1$)
we know that~${\mathcal{J}}^{\Omega_r}(\cdot,T)$ is decreasing in $s\in \left(0,\frac{1}{2}\right]$.

Therefore, when~$T$ is sufficiently large and~$r$ is sufficiently small,
the minimum of~${\mathcal{J}}^{\Omega_r}(\cdot,T)$ must occur for~$s\in\left[\frac12,\sigma_T\right]$.
This proves~\eqref{hbre4352618}.

Also, the claim in~\eqref{gllpldmtr} follows from~\eqref{gllpldmtr-2}.
\end{proof}

One is now in the position of completing the proof of Corollary~\ref{pfff}. For the reader's convenience, this technical proof is deferred to Appendix~\ref{OJldnfeABBpldfebfBas5Lr}.

\begin{proof}[Proof of Theorem~\ref{0ojrt9-43hgoiuf6hb7v65yTHMSMD}] 
We begin by observing that since the integrand in~\eqref{0ojrt9-43hgoiuf6hb7v65yc70bv65bvcnonyTDv8y4bv45ybv} is nonnegative, then we can change the order of integration and obtain
\begin{equation}\label{OJSLndDERO}
\begin{split}
\int_{0}^T\int_{\R^n} e^{-t\left(2\pi\left|\xi \right|\right)^{2s}}\left|\widehat{p}(\xi) \right|^2\,d\xi\,dt&=\int_{\R^n}\int_{0}^T e^{-t\left(2\pi\left|\xi \right|\right)^{2s}}\left|\widehat{p}(\xi) \right|^2\,dt\,d\xi\\
&=\int_{\R^n} \frac{1-e^{-T\left(2\pi\left|\xi \right|\right)^{2s}}}{\left(2\pi\left|\xi \right|\right)^{2s}}\left|\widehat{p}(\xi) \right|^2\,d\xi.
\end{split}
\end{equation}
Now, we define the function
\begin{equation*}
f(\xi,s,T):=\frac{1-e^{-T\left(2\pi\left|\xi \right|\right)^{2s}}}{\left(2\pi\left|\xi \right|\right)^{2s}},
\end{equation*}
and we compute its partial derivative with respect to the fractional parameter $s$ 
\begin{equation}\label{KAMSN012owe}
\begin{split}
\partial_s f(\xi,s,T) & =\frac{2\,Te^{-T\left(2\pi\left|\xi \right|\right)^{2s}}\ln\left(2\pi\left|\xi \right|\right) \left(2\pi\left|\xi \right|\right)^{4s}- \left(1-e^{-T\left(2\pi\left|\xi \right|\right)^{2s}}\right)2 \ln\left(2\pi\left|\xi \right|\right) \left(2\pi\left|\xi \right|\right)^{2s}}{\left(2\pi\left|\xi \right|\right)^{4s}}\\  
&=2 \ln\left(2\pi\left|\xi \right|\right)\left(2\pi\left|\xi \right|\right)^{2s} \left(\frac{T \left(2\pi\left|\xi \right|\right)^{2s}e^{-T\left(2\pi\left|\xi \right|\right)^{2s}}-1+e^{-T\left(2\pi\left|\xi \right|\right)^{2s}}}{\left(2\pi\left|\xi \right|\right)^{4s}} \right)\\ 
&=T^2\,k(\xi,s)\;\sigma'\left(T\left(2\pi\left|\xi \right|\right)^{2s}\right),
\end{split}
\end{equation}
where $\sigma$ is the function given by
\begin{equation}\label{mo-23refvchguq00}
\sigma(x):=\frac{1-e^{-x}}{x}
\end{equation}
and we set
\begin{equation}\label{gvoqrecxoq01}
k(\xi,s):=2\ln\left(2\pi\left|\xi \right|\right)\left(2\pi\left|\xi \right|\right)^{2s}.
\end{equation}
We notice that $k(\xi,s)\in L^\infty\left(B_1\times (\beta,1)\right)$ for each $\beta\in (0,1)$. 

Also, 
$$\lim_{x\to0}\sigma'(x)=-\frac12,\qquad\lim_{x\to+\infty}\sigma'(x)=0\qquad{\mbox{and}}\qquad e^x\ge x+1.$$
{F}rom these observations, it follows that
\begin{equation}\label{knyreclao3}
\sigma'(x)\leq 0\quad{\mbox{ for each $x\in (0,+\infty)$}} \qquad{\mbox{and}}\qquad \sigma' \in L^\infty((0,+\infty)).
\end{equation}
Therefore, if $\beta\in (0,1)$, we obtain the estimate
\begin{equation}\label{10293eforeachsbeta1}\sup_{{\xi\in B_1}\atop{s\in (\beta,1)}}
\left|\partial_s f(\xi,s,T) \right|\leq T^2\| k\|_{L^\infty(B_1\times(\beta,1))}\|\sigma'\|_{L^\infty((0,+\infty))}
=:C_{T,\beta}.
\end{equation}

Furthermore, we notice that we can rewrite the second line in~\eqref{KAMSN012owe} as
\begin{equation*}
\partial_s f(\xi,s,T)=-2u(\xi,s)v\left(T \left(2\pi\left|\xi \right|\right)^{2s}\right),
\end{equation*}
where
\begin{equation*}
v(x):=1-e^{-x}(x+1)\quad\mbox{and}\quad u(\xi,s):=\frac{\ln\left(2\pi\left|\xi \right|\right)}{\left(2\pi\left|\xi \right|\right)^{2s}}.
\end{equation*}
We note that~$v\in L^\infty((0,+\infty))$. Also,
for each $\beta \in (0,1)$,
we see that~$u\in L^\infty\left(\left(\R^n\setminus B_1\right) \times (\beta,1)\right)$.
As a consequence, for each $s\in (\beta,1)$,
\begin{equation*}
\sup_{{\xi \in \R^n\setminus B_1}\atop{\beta \in (0,1)}}\left|\partial_s f(\xi,s,T)\right|\leq \|u\|_{L^\infty( (\R^n\setminus B_1) \times (\beta,1) )}
\|v\|_{L^\infty((0,+\infty))} =:\tilde{C}_\beta.
\end{equation*}

In view of this estimate and~\eqref{10293eforeachsbeta1}, we conclude that for each $\beta \in (0,1)$ there exists some constant~$\widehat{C}_{T,\beta}$, depending on $\beta$ and $T$, such that,
for every~$\xi\in\R^n$ and~$s\in(\beta,1)$,
\begin{equation*}
\left|\partial_s f(\xi,s,T)\right| \left|\widehat{p}(\xi)\right|^2 \leq \widehat{C}_{T,\beta}\left|\widehat{p}(\xi)\right|^2.
\end{equation*} 
Accordingly, we can apply the Dominated Convergence Theorem, and obtain that, for every~$s\in(0,1)$, 
\begin{equation}\label{knoq-12345}
\partial_s \int_{\R^n} \frac{1-e^{-T\left(2\pi\left|\xi \right|\right)^{2s}}}{\left(2\pi\left|\xi \right|\right)^{2s}}\left|\widehat{p}(\xi) \right|^2\,d\xi= \int_{\R^n} T^2\,k(\xi,s)\;\sigma'\left(T\left(2\pi\left|\xi \right|\right)^{2s}\right)\left|\widehat{p}(\xi) \right|^2\,d\xi.
\end{equation}
Now, we notice that $k(\xi,s)$ is negative if and only if $\xi\in B_{\frac{1}{2\pi}}$. From this observation, \eqref{OJSLndDERO} and~\eqref{knyreclao3}, we can conclude the proof of Theorem~\ref{0ojrt9-43hgoiuf6hb7v65yTHMSMD}.
\end{proof}    

\begin{rem}{\rm We observe that the result stated in Theorem~\ref{0ojrt9-43hgoiuf6hb7v65yTHMSMD} is optimal, in the sense that translations of~$\widehat{p}$ supported in $B_{\frac{1}{2\pi}}$ produce internal critical points in $[0,1]$, breaking the monotonicity of~\eqref{0ojrt9-43hgoiuf6hb7v65yc70bv65bvcnonyTDv8y4bv45ybv} with respect to the L\'evy exponent (in particular, one cannot expect that the efficiency functional is always monotone in~$s$). Indeed, if for some~$v\in \partial B_1$ and some~$\widehat{p}\in L^2(\R^n)$
supported in~$B_{\frac{1}{2\pi}}$ we define the function
\begin{equation*}
g(s,r):=\int_{\R^n}
T\,k(\xi,s)\;\sigma'\left(T\left(2\pi\left|\xi \right|\right)^{2s}\right)\,|\widehat p(\xi-rv)|^2\,d\xi\,dt,
\end{equation*}
where $\sigma$ and~$k$ are given respectively in~\eqref{mo-23refvchguq00} and~\eqref{gvoqrecxoq01}, then by~\eqref{OJSLndDERO}, \eqref{knoq-12345}, and Theorem~\ref{0ojrt9-43hgoiuf6hb7v65yTHMSMD} we have that $g(s,0)>0$, while $g(s,\frac{3}{2\pi})<0$ for each $s\in (0,1)$. Since $g(s,\cdot)\in C((0,+\infty))$ and~$g$ is the derivative of the efficiency functional in~\eqref{0ojrt9-43hgoiuf6hb7v65yc70bv65bvcnonyTDv8y4bv45ybv},
we have that for each $s_0\in (0,1)$ there exists some~$r_0\in (0,\frac{3}{2\pi})$ such that $s_0$ is a critical point of ~\eqref{0ojrt9-43hgoiuf6hb7v65yc70bv65bvcnonyTDv8y4bv45ybv} for $\widehat{p}(\xi-rv)$. 

As an example of this phenomenon we consider the translations of $\widehat{p}(\xi)=\chi_{B_{\frac{1}{3\pi}}}(\xi)$. In particular, in Figure~\ref{bhcaomcte45bc01} we have plotted the two variables function 
$$
L(s,r):=\int_{0}^1\int_{\R}e^{-(2\pi |\xi|)^{2s}t}\,|\chi_{B_{\frac{1}{3\pi}}}(\xi-r)|^2\,d\xi\,dt. 
$$
We can clearly see that as we translate the support of~$\widehat{p}_r(\xi):=\chi_{B_{\frac{1}{3\pi}}}(\xi-r)$, the monotonicity in~$s$ of~$L(s,r)$ changes, according to Theorem~\ref{0ojrt9-43hgoiuf6hb7v65yTHMSMD}.}\end{rem}

\begin{center}
\begin{figure}[!ht]
\includegraphics[width=0.40\textwidth]{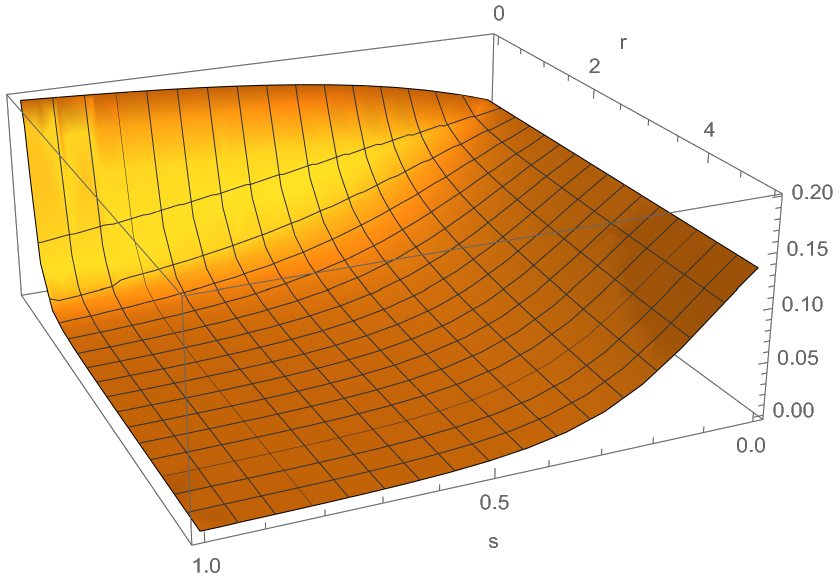}
\caption{\sl\footnotesize Plot of $ (0,1)\times (0,5)\ni (s,r)\to L(s,r)$. As $r$ increases we see that the monotonicity of $L(\cdot,r)$ changes. }
\label{bhcaomcte45bc01}
\end{figure}
\end{center}

\section{Conclusions}

Like all significant scientific theories, the L\'evy flight foraging hypothesis has undergone continuous questioning, discussion, debate, and revision. This is particularly due to the inherent challenge of distinguishing various factors that might contribute to anomalous diffusion, the influence of technological limitations on data collection and interpretation, and the sometimes nuanced differentiation between individual and collective behaviors.

In Section~\ref{DIAsdff} we have recalled the different views emerged on the L\'evy flight foraging hypothesis and also discussed the controversies about the use of distributions with infinite moment and processes with infinite propagation velocity.
\medskip

The main results of this paper have been presented in plain English in Section~\ref{1-Sljn3hfig} and in full mathematical detail in Section~\ref{forma}. Their interest lies in the attempt of providing a rigorous quantitative framework for a simple
model related to optimal strategies of foraging patterns. The model considered resembles that of 
search within patches and takes into account foragers diffusing through a fractional heat equation.
The forager can pick the L\'evy exponent of this diffusion in order to maximize the encounters with a given distribution of preys.

Models of this type, relying on a continuous version of L\'evy flights, may be conceptually advantageous, since they can reduce the number of external parameters involved in the analysis, allow explicit mathematical treatments, help understanding the core features
of the model, and favor the introduction of further realistic characteristics as a generalization of the system considered to start with.
Moreover, in spite of its structural simplicity, the model is rich enough to exhibit complicated behaviors which can be confronted with biological data and experiments. Additionally, these models can be utilized in the field of artificial intelligence,
e.g. in the study of robot behavior and optimization algorithms.\medskip

Models of this type have been introduced and studied in~\cite{EFLFFH}, where suitable efficiency functionals
have been put in place and analyzed in one spatial dimension, in~\cite{DGV1, DGV1-1, SPETTRALE}, where similar models have been investigated
in bounded domains in the light of spectral analysis, and in~\cite{ULTI}, where the notion of optimal strategy has been confronted with its factual feasibility.\medskip

In this paper, we provide three new main results, namely:
\begin{itemize}
\item Theorem~\ref{gbearacdf43s0}, confirming, as a sanity check, that, in this model, the foragers 
attaining the highest efficiency in situations in which the prey is available at the initial location are the stationary ones,
\item Theorem~\ref{kkjyyeffcdf5128}, stating that Gaussian strategies are optimal for large domains and very large searching times
(with a finer and quantitative discussion also in terms of the spatial dimension),
\item Theorem~\ref{0ojrt9-43hgoiuf6hb7v65yTHMSMD}, affirming that
if the Fourier transform of the distribution of a stationary prey is supported in a suitable volume, then this efficiency functional is monotone increasing in~$s$
and therefore the optimal strategy is given by the Gaussian dispersal~$s=1$, while
if the Fourier transform of the distribution of a stationary prey is supported in the complement of a suitable volume, then the efficiency functional is monotone decreasing in~$s$
and therefore the optimal strategy is given by~$s=0$.\end{itemize}

We stress that, roughly speaking, compactly supported Fourier transforms correspond to distributions of preys with controlled oscillation, while Fourier transforms supported close to infinity give rise to very oscillatory distributions of preys.

The precise statements of these results were given in Section~\ref{forma} and confronted with their biological interpretation.

We also mention that our model can be suitably modified to include the case of 
foragers possessing a limited lifespan (meaning that they must consume at least a specific quantity of food within a certain time frame, or die): we will specifically address this point in the forthcoming work~\cite{SPAN}.

\begin{appendix}

\section{Proof of Theorem~\ref{kkjyyeffcdf5128}}\label{APKmddu0oj4g034JLA}
For the sake of simplicity in what follows we will denote by $C$, up to renaming it, constants depending only on  $\Omega$ and $n$.
We claim that for each $r\in (0,+\infty)$ it holds that 
\begin{equation}\label{gre-yner4=99}
{\mathcal{J}}^{\Omega_r}(s,T)=
\frac1{r^{n}}
{\mathcal{J}}^{\Omega}\left(s,\frac{T}{r^{2s}}\right)
.
\end{equation}
To prove this, we observe that, using the scaling properties of~$G^s(t,x,y)$
(see e.g.~\cite[Difference~2.10]{MR3967804}), 
\begin{equation*}
\begin{split}
{\mathcal{J}}^{\Omega_r}(s,T)&=\frac{1}{T\left|\Omega\right|^2 r^{2n}}\int_{0}^T\int_{\Omega_r\times\Omega_r}G^s(t,x,y)\,dx\,dy\,dt\\
&=\frac{1}{T\left|\Omega\right|^2 }\int_{0}^T\int_{\Omega\times\Omega}G^s(t,rx,ry)\,dx\,dy\,dt\\
&=\frac{r^{-n}}{T\left|\Omega\right|^2 } \int_{0}^T\int_{\Omega\times\Omega}G^s\left(\frac{t}{r^{2s}},x,y\right)\,dx\,dy\,dt\\
&=\frac{r^{2s-n}}{T\left|\Omega\right|^2 } \int_{0}^{\frac{T}{r^{2s}}}\int_{\Omega\times\Omega}G^s(t,x,y)\,dx\,dy\,dt\\
&=\frac1{r^{n}}
{\mathcal{J}}^{\Omega}\left(s,\frac{T}{r^{2s}}\right),
\end{split}
\end{equation*}which leads to~\eqref{gre-yner4=99}.

We define the function
\begin{equation*}
g(s,r,T):=\int_{\mathbb{R}^n} \sigma\left(Tr^{-2s}(2\pi\left|\xi \right|)^{2s}\right)\left|\widehat{\chi}_{\Omega}(\xi)\right|^2\,d\xi,
\end{equation*}
where we have set
\begin{equation*}
\sigma(x):=\frac{1-e^{-x}}{x}.
\end{equation*}
We notice that
\begin{equation}\label{nhqwp01}
{\mathcal{J}}^{\Omega_r}(s,T)=\frac{1}{r^n} g(s,r,T).
\end{equation}
To check this,
we use~\eqref{0ojrt9-43hgoiuf6hb7v65yc70bv65bvcnonyTDv8y4bv45ybv} and~\eqref{gre-yner4=99} to see that
\begin{equation}\begin{split}\label{fdscfre333}
& {\mathcal{J}}^{\Omega_r}(s,T)=
\frac1{r^n}
{\mathcal{J}}^{\Omega}\left(s,\frac{T}{r^{2s}}\right)
=\frac{r^{2s-n}}{T}
\int_{0}^{T/r^{2s}}\int_{\R^n} e^{-(2\pi |\xi|)^{2s}t}\,|\widehat \chi_\Omega(\xi)|^2
\,d\xi\,dt\\
&\qquad=
\frac{r^{2s-n}}{T}
\int_{\R^n} \frac{1-e^{-Tr^{-2s}(2\pi |\xi|)^{2s}}}{(2\pi|\xi|)^{2s}}
\,|\widehat \chi_\Omega(\xi)|^2\,d\xi=\frac1{r^n}g(s,r,T),
\end{split}\end{equation} 
which establishes~\eqref{nhqwp01}.

As a consequence of~\eqref{nhqwp01},
the monotonicity properties of~${\mathcal{J}}^{\Omega_r}(\cdot,T)$ coincide with the ones of~$g(\cdot,r,T)$.
Hence, from now on we will focus on the monotonicity of the function~$g(s,r,T)$
with respect to~$s$.

For this, we compute the derivative 
\begin{equation*}
\begin{split}
\partial_s\Big(\sigma\left(Tr^{-2s}(2\pi\left|\xi \right|)^{2s}\right) \Big)&=\sigma'\left(Tr^{-2s}(2\pi\left|\xi \right|)^{2s}\right)
\partial_s\big(Tr^{-2s}(2\pi\left|\xi \right|)^{2s}\big)\\
&=\sigma'\left(Tr^{-2s}(2\pi\left|\xi \right|)^{2s}\right)\cdot 2Tr^{-2s}(2\pi\left|\xi \right|)^{2s}\ln\left(\frac{2\pi \left|\xi\right|}{r}\right)\\
&=2\,\theta\left(Tr^{-2s}(2\pi\left|\xi \right|)^{2s}\right)\ln\left(\frac{r}{2\pi \left|\xi\right|}\right),
\end{split}
\end{equation*}
where we have denoted 
\begin{equation}\label{defthetacvgt}
\theta(x):=-\sigma'(x)x=\frac{1-e^{-x}(x+1)}{x}.
\end{equation}
We also observe that
\begin{equation}\label{r043vfvdhhdfsHHHHHH}
\left\|\widehat{\chi}_{\Omega}\right\|_{L^2(\Omega)}^2=\left|\Omega\right|\qquad {\mbox{and}}\qquad
\left\|\widehat{\chi}_{\Omega}\right\|_{L^\infty(\R^n)}\leq \left|\Omega\right|.
\end{equation}
Using these facts and Theorem~2 in~\cite{MR142978}, we have that
\begin{equation}\label{bgtepqncpoz-7743v}
\begin{split}
&\int_{\R^n} \left|\ln\left(\frac{2\pi \left|\xi\right|}{r}\right)\right|\left|\widehat{\chi}_{\Omega}(\xi)\right|^2\,d\xi\\
=\;&\int_{\R^n} \Big(\left|\ln\left(2\pi \left|\xi\right|\right)\right|+\ln(r)\Big)\left|\widehat{\chi}_{\Omega}(\xi)\right|^2\,d\xi\\
=\;&
\sum_{k=0}^{+\infty} \int_{B_{2^{k+1}}\setminus B_{2^k}}
\left|\ln\left(2\pi \left|\xi\right|\right)\right|\,
\left|\widehat{\chi}_{\Omega}(\xi)\right|^2\,d\xi+
\int_{B_1} \left|\ln\left(2\pi \left|\xi\right|\right)\right|\,
\left|\widehat{\chi}_{\Omega}(\xi)\right|^2\,d\xi
+|\ln r|\int_{\R^n} \left|\widehat{\chi}_{\Omega}(\xi)\right|^2\,d\xi\\
\leq\; & \sum_{k=0}^{+\infty} \int_{B_{2^{k+1}}\setminus B_{2^k}} \frac{\left|\ln( 2\pi\left| \xi\right|)\right|}{\left|\xi\right|^2}\left|\xi\right|^2\left|\widehat{\chi}_{\Omega}(\xi)\right|^2\,d\xi+\int_{B_1}\left|\ln( 2\pi\left| \xi\right|)\right|\left|\widehat{\chi}_{\Omega}(\xi)\right|^2\,d\xi+\ln r\left|\Omega\right|
\\
\leq &\;\sum_{k=0}^{+\infty}  \frac{\ln(\pi 2^{k+2})}{2^{2k}}\int_{B_{2^{k+1}}\setminus B_{2^k}} \left|\xi\right|^2\left|\widehat{\chi}_{\Omega}(\xi)\right|^2\,d\xi+\left|\Omega\right|^2\int_{B_1}\left|\ln( 2\pi\left| \xi\right|)\right|\,d\xi+\ln r\left|\Omega\right|\\
\leq  &\;\sum_{k=0}^{+\infty} \frac{{C}\ln(\pi 2^{k+2})}{2^{2k}}2^{k+1}+ {C}\left|\Omega\right|^2+\ln r\left|\Omega\right|\\
=& \; {C}\sum_{k=0}^{+\infty} \frac{\ln\pi
+(k+2)\ln 2}{2^{k-1}}+{C}\left|\Omega\right|^2+\ln r\left|\Omega\right|\\
=&\; {C}+  \ln r\left|\Omega\right|.
\end{split}
\end{equation}
Furthermore, we notice that $\theta \in L^\infty((0,+\infty))$. Thanks to this  and~\eqref{bgtepqncpoz-7743v}, we can apply the Dominated Convergence Theorem and conclude that 
\begin{equation*}
\partial_s g(s,r,T)=2\int_{\R^n}\theta\left(Tr^{-2s}(2\pi\left|\xi \right|)^{2s}\right)\ln\left(\frac{r}{2\pi \left|\xi\right|}\right)\left|\widehat{\chi}_{\Omega}(\xi)\right|^2\,d\xi.
\end{equation*}
We write~$\partial_s g(s,r,T)=2(I+II)$, where~$I$ and~$II$ are defined as
\begin{equation}\begin{split}\label{so038v5b90riwoIIII}
&I:=\int_{B_{\frac{r}{2\pi}}} \theta\left(Tr^{-2s}(2\pi\left|\xi \right|)^{2s}\right)\ln\left(\frac{r}{2\pi \left|\xi\right|}\right)\left|\widehat{\chi}_{\Omega}(\xi)\right|^2\,d\xi\\
{\mbox{and }}\qquad &II:=\int_{\R^n\setminus B_{\frac{r}{2\pi}}} \theta\left(Tr^{-2s}(2\pi\left|\xi \right|)^{2s}\right)\ln\left(\frac{r}{2\pi \left|\xi\right|}\right)\left|\widehat{\chi}_{\Omega}(\xi)\right|^2\,d\xi.
\end{split}\end{equation}

Now, we prove~\eqref{hnbg512dscxpae}. Since  $\ln\left(\frac{r}{2\pi \left|\xi\right|}\right)\geq 0$ for each $\xi\in B_{\frac{r}{2\pi}}$ and $\theta\geq 0$, then 
\begin{equation}\label{mlare43..8}
\theta\left(Tr^{-2s}(2\pi\left|\xi \right|)^{2s}\right)\ln\left(\frac{r}{2\pi \left|\xi\right|}\right)\left|\widehat{\chi}_{\Omega}(\xi)\right|^2\geq 0
\end{equation}
for $\xi \in B_{\frac{r}{2\pi}}$. In particular, if $r\geq 2\pi$ this means that 
\begin{equation*}
\begin{split}
I &\geq \int_{B_{1}\setminus B_{\frac{1}{2}}} \theta\left(Tr^{-2s}(2\pi\left|\xi \right|)^{2s}\right)\ln\left(\frac{r}{2\pi \left|\xi\right|}\right)\left|\widehat{\chi}_{\Omega}(\xi)\right|^2\,d\xi.
\end{split}
\end{equation*}
Moreover, we notice that
\begin{equation}\label{gbre}
\theta(x)\leq \frac{1}{x},
\end{equation}
for each $x\in (0,+\infty)$, and also $\theta(x)>\frac{1}{2x}$ if $x\in (0,+\infty)$ is large enough. Thanks to this lower bound we have
\begin{equation}\label{dji35u34i5y89567ytredvcgehjd}
\begin{split}
I &\geq \int_{B_{1}\setminus B_{\frac{1}{2}}} \frac{r^{2s}}{2T(2\pi \left|\xi\right|)^{2s}}\ln\left(\frac{r}{2\pi \left|\xi\right|}\right)\left|\widehat{\chi}_{\Omega}(\xi)\right|^2\,d\xi\\
&\geq \frac{r^{2s}\ln r}{4T} \int_{B_{1}\setminus B_{\frac{1}{2}}} \frac{\left|\widehat{\chi}_{\Omega}(\xi)\right|^2}{(2\pi \left|\xi\right|)^{2s}}\,d\xi \\
&\geq {C}\,\frac{r^{2s}\ln r}{T}.
\end{split}
\end{equation}
for $r\in (1,+\infty)$ large enough and $T\geq r^2$.

Now, we estimate $II$. To do so, we observe that in~$\R^n\setminus
B_{\frac{r}{2\pi}}$ it holds that~$\ln \left(\frac{r}{2\pi \left|\xi\right|}\right)\leq 0$, and thus, to ease notation and identify positive quantities, we write
\begin{equation*}
II=-\int_{\R^n\setminus B_{\frac{r}{2\pi}}} \theta\left(Tr^{-2s}(2\pi\left|\xi \right|)^{2s}\right)\ln \left(\frac{2\pi \left|\xi\right|}{r}\right)\left|\widehat{\chi}_{\Omega}(\xi)\right|^2\,d\xi,
\end{equation*}
with 
\begin{equation}\label{bgtreclaor}
 \theta\left(Tr^{-2s}(2\pi\left|\xi \right|)^{2s}\right)\ln \left(\frac{2\pi \left|\xi\right|}{r}\right)\geq 0.
\end{equation}
Therefore, using~\eqref{gbre} we can estimate
\begin{equation*}
\begin{split}
&\int_{\R^n\setminus B_{\frac{r}{2\pi}}} \theta\left(Tr^{-2s}(2\pi\left|\xi \right|)^{2s}\right)\ln \left(\frac{2\pi \left|\xi\right|}{r}\right)\left|\widehat{\chi}_{\Omega}(\xi)\right|^2\,d\xi\\
\leq & \int_{\R^n\setminus B_{\frac{r}{2\pi}}} \frac{r^{2s}}{T(2\pi\left|\xi \right|)^{2s}}\ln \left(\frac{2\pi \left|\xi\right|}{r}\right)\left|\widehat{\chi}_{\Omega}(\xi)\right|^2\,d\xi\\
\leq \,&\frac{1}{T}\int_{\R^n\setminus B_{\frac{r}{2\pi}}} \ln\left(\frac{2\pi\left|\xi \right|}{r}\right) \left|\widehat{\chi}_{\Omega}(\xi)\right|^2\,d\xi.
\end{split}
\end{equation*}
Now, applying Theorem~2 in~\cite{MR142978} we obtain that 
\begin{equation*}
\begin{split}
&\int_{\R^n\setminus B_{\frac{r}{2\pi}}} \ln\left(\frac{2\pi\left|\xi \right|}{r}\right) \left|\widehat{\chi}_{\Omega}(\xi)\right|^2\,d\xi\\
=\;& \sum_{k=0}^{+\infty} \int_{B_{\frac{2^{k}r}{\pi}}\setminus B_{\frac{2^{k-1}r}{\pi}} } \ln\left(\frac{2\pi\left|\xi \right|}{r}\right)\frac{\left|\xi \right|^2}{\left|\xi \right|^2}\left|\widehat{\chi}_{\Omega}(\xi)\right|^2\,d\xi\\
\leq\; &\sum_{k=0}^{+\infty} \frac{\pi^2(k+1)\ln 2}{r^2 2^{2k-2}}
\int_{B_{\frac{2^{k}r}{\pi}}\setminus B_{\frac{2^{k-1}r}{\pi}} }\left|\xi \right|^2\left|\widehat{\chi}_{\Omega}(\xi)\right|^2\,d\xi\\
\leq \;&C \sum_{k=0}^{+\infty} \frac{\pi^2(k+1)\ln 2}{r^2 2^{2k-2}} \cdot
\frac{2^kr}{\pi}\\
= \;& \frac{{C}}{r}.
\end{split}
\end{equation*}

Gathering these observations, we deduce that 
\begin{equation*}
II \geq -\frac{{C}}{r\,T}.
\end{equation*}
{F}rom this and~\eqref{dji35u34i5y89567ytredvcgehjd},
we infer that, if~$r\in (1,+\infty)$ is large enough and~$T\in [r^2,+\infty)$,
\begin{equation*}
\partial_s g(s,r,T)=2(I+II) \geq \frac{{C}}{T}\left(r^{2s}\ln r- \frac{{\tilde{C}}}{r}\right)>0.
\end{equation*}
Recalling~\eqref{nhqwp01}, we complete the proof of~\eqref{hnbg512dscxpae}

Now, we show~\eqref{gnapegcnr30d}. To do so, we recall the definition of
the function~$\theta$ in~\eqref{defthetacvgt} and
we observe that, for each~$x\in (0,+\infty)$, 
\begin{equation}\label{inveiowr55p0987698}
\theta(x)\leq \widehat{C}\,x,
\end{equation}
for some positive constant $\widehat{C}$. 

We recall the definition of~$I$ in~\eqref{so038v5b90riwoIIII} and we change variable to see that 
\begin{equation}\label{knolqw31vz885}
I= \frac{r^n}{(2\pi)^n}\int_{B_1} \theta\left(T\left|\eta \right|^{2s}\right)\ln\left(\frac{1}{\left|\eta\right|}\right)\left|\widehat{\chi}_{\Omega}\left(\frac{r\eta}{2\pi}\right)\right|^2\,d\eta.
\end{equation}
In view of~\eqref{r043vfvdhhdfsHHHHHH} and~\eqref{inveiowr55p0987698},
we can estimate
\begin{equation*}
\begin{split}
I\leq &\;\frac{\widehat{C} r^n}{(2\pi)^{n}}\int_{B_{1}} T\left|\eta \right|^{2s} \ln\left(\frac{1}{\left|\eta\right|}\right) \left|\Omega \right|^2\,d\eta\\
\leq &\; \frac{\widehat{C}\,r^n T \left|\Omega \right|^2}{(2\pi)^{n}}  \int_{B_{1}} \ln\left(\frac{1}{\left|\eta\right|}\right)\,d\eta\\
\leq & \; {C} r^n T. 
\end{split}
\end{equation*}

Also, recalling the definition of~$II$ in~\eqref{so038v5b90riwoIIII},
if $r\in (0,\pi)$ is sufficiently small, possibly in dependance of~$T$, we have that 
\begin{equation*}
\begin{split}
-II=&\;\int_{\R^n\setminus B_{\frac{r}{2\pi}}} \theta\left(Tr^{-2s}(2\pi\left|\xi \right|)^{2s}\right)\ln \left(\frac{2\pi \left|\xi\right|}{r}\right)\left|\widehat{\chi}_{\Omega}(\xi)\right|^2\,d\xi\\
\ge &\;\int_{B_1\setminus B_{1/2}} \theta\left(Tr^{-2s}(2\pi\left|\xi \right|)^{2s}\right)\ln \left(\frac{2\pi \left|\xi\right|}{r}\right)\left|\widehat{\chi}_{\Omega}(\xi)\right|^2\,d\xi\\
\geq & \;\int_{B_1\setminus B_{{1}/{2}}}\frac{r^{2s}}{T\,\left(2\pi\left|\xi \right|\right)^{2s}} \ln\left(\frac{\pi}{r}\right)\left|\widehat{\chi}_{\Omega}(\xi)\right|^2\,d\xi\\
=&\;\frac{r^{2s}}{T}\ln\left(\frac{\pi}{r}\right)\int_{B_1\setminus B_{{1}/{2}}} \frac{ \left|\widehat{\chi}_{\Omega}(\xi)\right|^2}{(2\pi\left| \xi\right|)^{2s}} \,d\xi\\
=& \;\frac{{C}\, r^{2s}}{T }\ln\left(\frac{\pi}{r}\right).
\end{split}
\end{equation*}

Therefore, we conclude that
\begin{equation*}
\begin{split}
\partial_s g(s,T,r)= 2(I+II)\leq  {C}r^{2s}\left({\tilde{C}}r^{n-2s} T-\frac{1}{T}\ln\left(\frac{\pi}{r}\right)\right),
\end{split}
\end{equation*} 
which gives that for $r>0$ small enough, depending on $\Omega$ and $T$, we have that $\partial_s g(s,r,T)<0$ as far as $n\geq 2s$. This completes
the proof of~\eqref{gnapegcnr30d}.

Now, we study the case $n<2s$, that is, $n=1$ and~$s\in\left(\frac12,1\right)$. 
For this, we notice that for each~$x\in (0,+\infty)$ there exists some constant~$C_x$ such that 
\begin{equation*}
\theta(y)\geq C_x y,
\end{equation*}
for each $y\in (0,x)$. Therefore, 
if~$|\eta|\in\left(\frac1{T^{\frac{1}{2s}}}, \frac2{T^{\frac{1}{2s}}}\right)$
then~$T|\eta|^{2s}\in(1, 2^{2s})\subseteq(1,4)$, and thus
$$\theta(T|\eta|^{2s})\geq {C} T|\eta|^{2s},$$
and moreover $2 T^{-\frac{1}{2s}} \leq 1$ for each $s\in (0,1)$ if $T\geq 4$. 

Also, we notice also that there exists some $\tilde{r}_\Omega\in (0,1)$ such that, for each $r\in (0,\tilde{r}_\Omega)$ and $\eta \in B_1 $, it holds that 
\begin{equation*}
\left|\widehat{\chi}_{\Omega}\left(\frac{r\eta}{2\pi}\right)\right| \geq\frac{\left|\Omega \right|}{2}.
\end{equation*}
As a consequence of this, recalling also~\eqref{knolqw31vz885} and using the Dominated Convergence Theorem, we deduce that,
if~$T\geq 4$ and $r\in (0,\tilde{r}_\Omega)$, then it holds that
\begin{equation*}
\begin{split}
I &\geq {C}r \int_{\frac1{T^{\frac{1}{2s}}}}^{\frac2{T^{\frac1{2s}}}}  T\left|\eta \right|^{2s}\ln\left(\frac{1}{\left|\eta\right|} \right)\left|\widehat{\chi}_{\Omega}\left(\frac{r\eta}{2\pi}\right)\right|^2\,d\eta
\geq  {C}r\ln\left(\frac{T}{4}\right)\frac{1}{T^{\frac1{2s}}}.
\end{split}
\end{equation*}

Furthermore, changing variables~$\eta:=\frac{2\pi\xi}r$ and then~$\omega:=\frac{\eta}{\left|\eta \right|^2}$ in~$II$, we see that
\begin{equation*}
\begin{split}
II=\;&-\frac{r^n}{(2\pi)^{n}} \int_{\R^n\setminus B_1} \theta\left(T\left|\eta \right|^{2s}\right)\ln\left(\left|\eta \right|\right)\left|\widehat{\chi}_{\Omega}\left(\frac{r\eta}{2\pi}\right)\right|^2\,d\eta\\
=\;&-\frac{r^n }{(2\pi)^{n}}
\int_{B_1} \theta\left(T\left|\omega \right|^{-2s}\right)\ln\left(\frac{1}{\left|\omega \right|}\right)\left|\widehat{\chi}_{\Omega}\left(\frac{r\omega}{2\pi \left|\omega \right|^2}\right)\right|^2\,\frac{d\omega}{\left|\omega \right|^2}.
\end{split}
\end{equation*}
Now, if $n=1$, $\sigma\in (1/2,1)$ and~$s\in (\sigma,1)$, we use~\eqref{r043vfvdhhdfsHHHHHH} and~\eqref{gbre} to estimate
\begin{equation}\label{kkktttrrrecfr}\begin{split}
-II\leq\;& \frac{|\Omega|^2 r}{2\pi}
\int_{B_1}\frac{\left|\omega\right|^{2s}}{T} \ln\left(\frac1{\left|\omega \right|}\right)\,\frac{d\omega}{\left|\omega \right|^2}
\\= \;&\frac{|\Omega|^2 r}{\pi\,T}
\int_0^{1} \omega^{2s-2} \ln\left(\frac1{\omega}\right)\,d\omega
\\
=\;& \frac{|\Omega|^2 r}{\pi\,T(2s-1)^2}\\
\leq\;&\frac{|\Omega|^2 r}{\pi\,T(2\sigma-1)^2}.
\\
\end{split}
\end{equation}
Therefore, if~$T\geq 4$ and $r\in (0,\tilde{r}_\Omega)$, making use also of~\eqref{fdscfre333} we obtain that  
\begin{equation}\label{ooorefdvcte}
\begin{split}
\partial_s \mathcal{J}^{\Omega_r}(s,T)&=\frac{1}{r} \partial_s g(s,r,T)\\ 
&=\frac{2}{r}(I+II)\\
&\geq \frac{1}{T}\left({C}\ln\left(\frac{T}{4}\right)T^{1-\frac{1}{2s}}-\frac{|\Omega|^2 }{\pi\,(2\sigma-1)^2}\right).
\end{split}
\end{equation}
As a result, there exists some $T_\sigma \in (4,+\infty)$ such that for each $T\in (T_\sigma,+\infty)$, $r\in (0, \tilde{r}_\Omega)$ and~$s\in (\sigma,1)$ it holds that
\begin{equation*}
\partial_s\mathcal{J}^{\Omega_r}(s,T) >0
\end{equation*}
This completes the proof of~\eqref{fverd..34}.~\hfill$\Box$

\section{Proof of Corollary~\ref{pfff}}\label{OJldnfeABBpldfebfBas5Lr}

We begin by observing that, if~$R>0$
\begin{equation}\label{uybguhg6643dre}
\begin{split}&
\widehat{\chi}_{(-R,R)}(\xi)=\int_{\R} e^{-2i \pi  \xi x}\chi_{(-R,R)}(x)\,dx=\int_{-R}^{R}\cos(2\pi \xi x)-i\sin(2\pi \xi x)\,dx\\&\qquad\qquad\qquad
=\frac{\sin(2\pi x \xi)}{2\pi \xi}\bigg|_{-R}^{R}=\frac{\sin(2\pi R\xi)}{\pi \xi}. 
\end{split}
\end{equation}

Besides, we notice that the value of $\mathcal{J}_\infty^\Omega(r,s)$ is invariant under translations of the domain $\Omega$. Indeed, if $v\in \R$, then 
\begin{equation*}
\widehat{\chi}_{\Omega+v}(\xi)=\int_{\Omega+v}e^{-2\pi i \xi x}\,dx=\int_{\Omega}e^{-2\pi i (x+v)\xi}\,dx=e^{-2\pi i v\xi}\widehat{\chi}_{\Omega}(\xi)
\end{equation*}
and therefore, by~\eqref{knrec53vdo97t},
\begin{eqnarray*} &&\mathcal{J}_\infty^{\Omega+v}(r,s)= r^{2s-1}\int_{\R}\frac{\left|\widehat{\chi}_{\Omega+v}(\xi) \right|^2}{(2\pi \left|\xi \right|)^{2s}}\,d\xi = r^{2s-1}\int_{\R}\frac{\left|
e^{-2\pi i v\xi}\widehat{\chi}_{\Omega}(\xi)\right|^2}{(2\pi \left|\xi \right|)^{2s}}\,d\xi\\&&\qquad\qquad\qquad\qquad
= r^{2s-1}\int_{\R}\frac{\left|
\widehat{\chi}_{\Omega}(\xi)\right|^2}{(2\pi \left|\xi \right|)^{2s}}\,d\xi=\mathcal{J}_\infty^{\Omega}(r,s).\end{eqnarray*}
For this reason, without loss of generality, we can assume that $\Omega=(-R,R)$ for some $R\in (0,+\infty)$.

Now we prove the limit in~\eqref{lmoetcbt5432}. For this, we recall~\eqref{fdscfre333} and see that 
\begin{equation*}
T\mathcal{J}^{\Omega_r}(s,T)=r^{2s-1}
\int_{\R} \frac{1-e^{-Tr^{-2s}(2\pi |\xi|)^{2s}}}{(2\pi|\xi|)^{2s}}
\,|\widehat \chi_\Omega(\xi)|^2\,d\xi.
\end{equation*}
Now, we notice that if $\beta \in \left(0,\frac{1}{2}\right)$, $s\in \left[0,\beta\right)$ and $T\in (0,+\infty)$, then 
\begin{equation*}
\frac{1-e^{-Tr^{-2s}\left(2\pi \left|\xi\right|\right)^{2s}}}{\left(2\pi \left|\xi \right| \right)^{2s}} \left| \widehat{\chi}_{\Omega}(\xi)\right|^2\leq \chi_{B_{\frac{1}{2\pi}}}(\xi)\frac{\left| \widehat{\chi}_{\Omega}(\xi)\right|^2}{\left(2\pi \left|\xi \right| \right)^{2\beta}} + \chi_{B_{\frac{1}{2\pi}}^c}(\xi)\left| \widehat{\chi}_{\Omega}(\xi)\right|^2 
.\end{equation*} 
In particular, since $\widehat{\chi}_{\Omega}\in L^2(\R)\cap L^\infty(\R)$, due to~\eqref{uybguhg6643dre},
we obtain that the function on the right-hand side is in $L^1(\R^n)$. Therefore, thanks to the Dominated Convergence Theorem, we can conclude that~\eqref{lmoetcbt5432} holds.

Now, we prove~\eqref{hbcyevco87tb} and~\eqref{uytvuytvuytvfruytf}. We observe that 
\begin{equation*}
\begin{split}
\widehat{\chi}_{\Omega}(\xi)=\int_{-R}^R e^{-2i \pi \xi x}\,dx=R \int_{-1}^1e^{-2 i \pi R \xi x}\,dx=R\widehat{\chi}_{(-1,1)}(R\xi),
\end{split}
\end{equation*}
and therefore the following scaling property holds true:
\begin{equation*}
\begin{split}
\mathcal{J}_\infty^{\Omega}(r,s)=r^{2s-1}\int_{\R}\frac{\left|\widehat{\chi}_{\Omega}(\xi) \right|^2}{\left(2\pi\left|\xi \right| \right)^{2s}}\,d\xi=R^{1+2s} r^{2s-1} \int_{\R} \frac{\left|\widehat{\chi}_{(-1,1)}(\xi) \right|^2}{\left(2\pi\left|\xi \right| \right)^{2s}}\,d\xi=R^{1+2s}\mathcal{J}_{\infty}^{(-1,1)}(r,s).
\end{split}
\end{equation*}
{F}rom this latter equation we evince that 
\begin{equation}\label{noiwrecxloiuyb}
\begin{split}
\partial_s \mathcal{J}_\infty^{\Omega}(r,s)&=2R^{1+2s}r^{2s-1}\left(\ln(R\,r)g(s)+\frac{1}{2}g'(s)\right),
\end{split}
\end{equation}
where we have set 
\begin{equation*}
g(s):=\int_{\R} \frac{\left| \widehat{\chi}_{(-1,1)}(\xi)\right|^2}{\left(2\pi \left|\xi \right| \right)^{2s}}\,d\xi.
\end{equation*}

Now, given~$m\in \N$ and~$\ell>0$, we use the notation $\ln^m\ell:=(\ln\ell)^m$ and we
observe that, for all~$\beta \in \left(0,\frac{1}{2}\right)$ 
and~$s\in [0,\beta)$,
\begin{equation*}
\ln^m\left(2\pi \left|\xi \right|\right)\frac{\left| \widehat{\chi}_{(-1,1)}(\xi)\right|^2}{\left(2\pi \left|\xi \right| \right)^{2s}}\leq  \ln^m\left(2\pi \left|\xi \right|\right)\left(  \chi_{B_{\frac{1}{2\pi}}}(\xi)\frac{\left| \widehat{\chi}_{(-1,1)}(\xi)\right|^2}{\left(2\pi \left|\xi \right| \right)^{2\beta}} + \chi_{B_{\frac{1}{2\pi}}^c}(\xi)\left| \widehat{\chi}_{(-1,1)}(\xi)\right|^2\right). 
\end{equation*}
Thus, recalling~\eqref{uybguhg6643dre}, we have that
\begin{equation*}
\ln^m\left(2\pi \left|\xi \right|\right)\left(  \chi_{B_{\frac{1}{2\pi}}}(\xi)\frac{\left| \widehat{\chi}_{(-1,1)}(\xi)\right|^2}{\left(2\pi \left|\xi \right| \right)^{2\beta}} + \chi_{B_{\frac{1}{2\pi}}^c}(\xi)\left| \widehat{\chi}_{(-1,1)}(\xi)\right|^2\right)\in L^1(\R).
\end{equation*}
Thanks to these observations, we can apply the Dominate Convergence Theorem and obtain that~$g\in C^2\left(\big[0,\frac{1}{2}\big)\right)$, and for each~$s\in \left[0,\frac{1}{2}\right)$,
\begin{equation}\label{lmcdfvchgbrt54i8654rvgv}\begin{split}&
g'(s)=-2\int_{\R}\frac{\left|\widehat{\chi}_{(-1,1)}(\xi) \right|^2}{(2\pi \left|\xi \right|)^{2s}}\ln(2\pi\left|\xi \right|)\,d\xi\\ \mbox{and}\quad\qquad& g''(s)=4\int_{\R}\frac{\left|\widehat{\chi}_{(-1,1)}(\xi) \right|^2}{(2\pi \left|\xi \right|)^{2s}}\ln^2(2\pi\left|\xi \right|)\,d\xi>0.\end{split}
\end{equation}

Furthermore, according to~\eqref{uybguhg6643dre} and~\eqref{lmcdfvchgbrt54i8654rvgv}, we can compute that 
\begin{equation}\label{lmrecSSaew}
\begin{split}&
-\frac{1}{4} g'(0)=\int_{0}^{+\infty} \frac{\sin^2(2\pi \xi)}{\left(\pi \xi\right)^2}\ln(2\pi \xi) \,d\xi=\frac{2}{\pi}\int_{0}^{+\infty}\frac{\sin^2(\eta)}{\eta^2}\ln(\eta)\,d\eta\\&\qquad\qquad\qquad\qquad=-\frac{1}{\pi^2}\left(-1+\gamma+\ln(2)\right)<0,
\end{split}
\end{equation} 
where $\gamma$ is the Euler-Mascheroni constant. 

Also, thanks to~\eqref{lmcdfvchgbrt54i8654rvgv} and~\eqref{lmrecSSaew} we obtain that 
\begin{equation*}
g'(s)=g'(0)+\int_{0}^s g''(s)\,ds>0.
\end{equation*}
Now we introduce the function
\begin{equation*}
f(s):=-\frac{g'(s)}{2g(s)}.
\end{equation*}
We observe that~$f$
is negative and, according to~\eqref{noiwrecxloiuyb},
\begin{equation}\label{JINFCONF-234}
\partial_s \mathcal{J}_\infty^{\Omega}(r,s)=2R^{1+2s}r^{2s-1}g(s)\left(\ln(R\,r)+\frac{g'(s)}{2g(s)}\right)=2R^{1+2s}r^{2s-1}g(s)\big(\ln(R\,r)-f(s)\big).\end{equation}
We also define 
\begin{equation*}
M:=\sup_{s\in \left(0,\frac{1}{2}\right)}f(s)\in(-\infty,0].
\end{equation*}
Therefore, by~\eqref{JINFCONF-234}, if 
\begin{equation*}
r> r_{\Omega}:= \frac{e^{M}}{R}\in(0,+\infty),
\end{equation*}
then, for every~$s\in\left(0,\frac{1}{2}\right)$,
\begin{eqnarray*}&&
\partial_s \mathcal{J}_\infty^{\Omega}(r,s)\ge2R^{1+2s}r^{2s-1}g(s)\big(\ln(R\,r)-M\big)>0
\end{eqnarray*}
and consequently
the functional $\mathcal{J}_\infty^{\Omega}(r,\cdot)$ is strictly increasing in~$s\in \left(0,\frac{1}{2}\right)$. This concludes the proof of~\eqref{hbcyevco87tb}.

Now we prove~\eqref{uytvuytvuytvfruytf}. First, using also the fact that $g(s)>0$ in $s\in \big[0,\frac{1}{2}\big)$, we notice that~$f\in C\left(\big[0,\frac{1}{2}\big)\right)$. Thus, for each~$\sigma\in \left(0,\frac{1}{2}\right)$, we can define the finite value
\begin{equation*}
m_\sigma:=\inf_{s\in (0,\sigma)} f(s)\in(-\infty,0],
\end{equation*}
and, by~\eqref{JINFCONF-234}, we obtain that if 
\begin{equation}\label{223BIS}
r< r_{\sigma,\Omega}:=\frac{e^{m_\sigma}}{R}\in (0,+\infty)
\end{equation}
then, for all~$s\in (0,\sigma)$,
$$ \partial_s \mathcal{J}_\infty^{\Omega}(r,s)\le2R^{1+2s}r^{2s-1}g(s)\big(\ln(R\,r)-m_\sigma\big)<0,$$ showing that~$\mathcal{J}_\infty^{\Omega}(r,\cdot)$ is strictly decreasing in~$s\in (0,\sigma)$. This last step concludes the proof of~\eqref{uytvuytvuytvfruytf}.

It is only left to show~\eqref{ukybgiygiy} and~\eqref{uytvuytvuytvfruytf2}. To this end, we notice that, for each $\sigma \in (0,\frac{1}{2})$,
\begin{equation}\label{hbrecdfgte43}
\begin{split}&
\int_{0}^{\sigma} f(s)\,ds=-\frac{1}{2}\int_{0}^{\sigma}\frac{g'(s)}{g(s)}\,ds=-\frac{1}{2}\int_{0}^{\sigma}\frac{d}{ds}\Big(\ln(g(s))\Big)\,ds=-\frac{1}{2}\Big(\ln(g(\sigma))-\ln(g(0))\Big).
\end{split}
\end{equation}
Since
\begin{equation*}
\lim_{\sigma\nearrow \frac{1}{2}}g(\sigma)=+\infty,
\end{equation*} 
we deduce from~\eqref{hbrecdfgte43} that 
\begin{equation*}
\lim_{\sigma\nearrow \frac{1}{2}}\int_{0}^{\sigma}f(s)\,ds=-\infty.
\end{equation*}
In particular, this implies that 
\begin{equation*}
\lim_{\sigma\nearrow \frac{1}{2}}f(\sigma)=-\infty
\end{equation*}
and thus 
\begin{equation}\label{yhrefd5543}
\lim_{\sigma\nearrow \frac{1}{2}}m_{\sigma}=-\infty.
\end{equation}
This and~\eqref{223BIS} lead to the desired claim in~\eqref{ukybgiygiy}.

It remains to prove~\eqref{uytvuytvuytvfruytf2}. For this, we take~$r\in (r_{\sigma,\Omega},+\infty)$ and observe that
$$
\ln(Rr)>\ln(R r_{\sigma,\Omega})=\ln(e^{m_\sigma})=m_{\sigma}=\inf_{s\in (0,\sigma)} f(s).
$$
Therefore, there exists~$s_\sigma\in(0,\sigma)$ such that~$\ln(Rr)>f(s_\sigma)$.
As a result, by~\eqref{JINFCONF-234},
$$ \partial_s \mathcal{J}_\infty^{\Omega}(r,s_\sigma)=
2R^{1+2s_\sigma}r^{2s_\sigma-1}g(s_\sigma)\big(\ln(R\,r)-f(s_\sigma)\big)>0,
$$
from which we obtain~\eqref{uytvuytvuytvfruytf2}, as desired.~\hfill$\Box$

\end{appendix}

\vfill
\end{document}